%% file: d_assim_improved_arxiv.tex
\documentclass[a4paper]{article}
\usepackage{amsmath,amssymb,color}
\newcommand{\bu}{\boldsymbol u}
\newcommand{\Om}{\Omega}

\newcommand{\bv}{\boldsymbol v}

\newcommand{\bw}{\boldsymbol w}
\newcommand{\bbu}{\boldsymbol b}

\newcommand{\btau}{\boldsymbol \tau}
\newcommand{\be}{\boldsymbol e}
\newcommand{\bvar}{\boldsymbol \varphi}

\newcommand{\bs}{\boldsymbol s}

\newcommand{\bff}{\boldsymbol f}

\usepackage{hyperref}

\input{data_assim_shared}

\ifpdf
\hypersetup{
  pdftitle={Uniform in time error estimates for a finite element method applied to a downscaling data assimilation algorithm for the Navier-Stokes equations},
  pdfauthor={B. Garc\'{\i}a-Archilla, J. Novo, and E. S. Titi}
}
\fi

\begin{document}

\maketitle

% REQUIRED

%\author{ Bosco
%Garc\'{\i}a-Archilla\thanks{Departamento de Matem\'atica Aplicada
%II, Universidad de Sevilla, Sevilla, Spain. Research is supported by
%Spanish MINECO under grant MTM2015-65608-P (bosco@esi.us.es)}
%  \and Julia Novo\thanks{Departamento de
%Matem\'aticas, Universidad Aut\'onoma de Madrid, Spain.  Research is supported
%by Spanish MINECO
%under grants MTM2013-42538-P (MINECO, ES) and MTM2016-78995-P (AEI/FEDER, UE) and VA024P17 (Junta de Castilla y Leon, ES) cofinanced by FEDER funds (julia.novo@uam.es)}
%\and Edriss S. Titi\thanks{Department of Mathematics, Texas A \& M University, College Station, TX 77843, USA, and Department of Computer Science and Applied Mathematics, Weizmann Institute of Science, Rehovot 76100, Israel. Research is supported in part by the ONR grant N00014-15-1-2333, the Einstein Stiftung/Foundation - Berlin,
%through the Einstein Visiting Fellow Program, and by the John Simon Guggenheim Memorial Foundation (titi@math.tamu.edu. edriss.titi@weizmann.ac.il)}}
%\date{March 2, 2019}
%
%\maketitle
\begin{abstract}
In this paper we analyze a finite element method applied to a continuous downscaling data assimilation algorithm for the numerical approximation of the two and three dimensional Navier-Stokes equations corresponding to given measurements on a coarse spatial scale. For representing the coarse mesh measurements we consider different types of interpolation operators including a Lagrange interpolant. We obtain uniform-in-time estimates for the error between a finite element approximation and the reference solution corresponding to the coarse mesh measurements. We consider both the case of a plain Galerkin method and a Galerkin method with grad-div stabilization. For the stabilized method we prove error bounds in which the constants do not depend on inverse powers of the viscosity. Some numerical experiments illustrate the theoretical results.
\end{abstract}



\section{Introduction}
Data assimilation refers to a class of techniques that combine experimental data and simulation in order to obtain better predictions in a physical system.
There is a vast literature on data assimilation methods, specially in the recent years
(see e.g., \cite{Asch_et_al_2016}, \cite{Daley_1993},
 \cite{Kalnay_2003}, \cite{Law_Stuart_Zygalakis}, \cite{Reich_Cotter_2015}, and the references
therein). One of these techniques is {\it nudging}, where a penalty term is added in order to drive the approximate solution towards coarse mesh or large scale spatial observations of the data. In a recent work~\cite{Az_Ol_Ti}, a new approach, known as continuous data assimilation,  is introduced for a large class of dissipative partial differential equations, including Rayleigh-B\'enard convection \cite{FJTi}, the planetary geostrophic ocean dynamics model \cite{FLTi}, etc.
(see also references therein). Continuous data assimilation has also been used in numerical studies, for example, with the Chafee-Infante reaction-diffusion equation  the Kuramoto-Sivashinsky
equation (in the context of feedback control)~\cite{Lunasin-Titi}, Rayleigh-B\'enard convection equations~\cite{Altaf_et_al}, \cite{FJJTi}, and the Navier-Stokes equations~\cite{Gesho-Olson-Titi}, \cite{HOTi}. However, there is much less numerical analysis of this technique. The present work concerns with the numerical analysis of continuous data assimilation for the Navier-Stokes equations when discretized with mixed finite element methods (MFE).
%Notably, similar results can be equally extended to other dissipative evolution equations such as the
%Rayleigh-B\'enard convection \cite{FJTi} and the planetary geostrophic ocean dynamics model \cite{FLTi}, etc... (see also references therein).

%In this work we focus on the Navier-Stokes equations
To be more precise, we consider the Navier-Stokes equations (NSE)
\begin{align}
\label{NS} \partial_t\bu -\nu \Delta \bu + (\bu\cdot\nabla)\bu + \nabla p &= \bff &&\text{in }\ (0,T]\times\Omega,\nonumber\\
\nabla \cdot \bu &=0&&\text{in }\ (0,T]\times\Omega,
\end{align}
in a bounded domain $\Omega \subset {\mathbb R}^d$, $d \in \{2,3\}$. In~\eqref{NS},
$\bu$ is the velocity field, $p$ the kinematic pressure, $\nu>0$ the kinematic viscosity coefficient,
 and $\bff$ represents the accelerations due to external body forces acting
on the fluid. The Navier-Stokes equations \eqref{NS} must be complemented with boundary conditions. For simplicity,
we only consider homogeneous
Dirichlet boundary conditions $\bu = \boldsymbol 0$ on $\partial \Omega$.

%homogeneous
%Dirichlet boundary conditions $\bu = \boldsymbol 0$ on $\partial \Omega$ will be considered in the present paper.

Following \cite{Mondaini_Titi} we consider given coarse spatial scale measurements, corresponding to a solution $\bu$ of \eqref{NS}, observed at a coarse spatial mesh. The measurements are assumed to be continuous in time and error-free. We denote by $I_H(\bu)$ the operator used for interpolating these
measurements, where $H$ denotes the resolution of the coarse spatial mesh.
Since the initial condition for $\bu$ is missing one cannot compute $\bu$ by simulating equation \eqref{NS} directly.
To overcome this difficulty it was suggested in~\cite{Az_Ol_Ti} to consider instead a solution~$\bv$ of the following approximating system
\begin{eqnarray}\label{eq:mod_NS}
 \partial_t\bv -\nu \Delta \bv + (\bv\cdot\nabla)\bv + \nabla \tilde p&=&\bff -\beta(I_H(\bv)-I_H(\bu)),\ \text{in }\ (0,T]\times\Omega,\nonumber\\
\nabla \cdot \bv&=&0, \ \text{in }\ (0,T]\times\Omega,
\end{eqnarray}
where $\beta$ is the relaxation (nudging) parameter.

In the case of the Navier-Stokes equations (and indeed, of many other nonlinear dissipative systems), it is well-known that
for relatively not so small Reynolds numbers, solutions are unstable and even chaotic. For this reason, it is expected that any small error in the initial data could lead to exponentially growing error in the solutions. Notably, the instabilities in the NSE occur at the large spatial scales, while the fine scales are stabilized by the viscosity. For this reason once the large spatial scales are stabilized, as it is done in the proposed downscaling data assimilation approximation, equation \eqref{eq:mod_NS},  the corresponding solution are stable and converge to the same solution~$\bu$ that is corresponding
to~$I_H(\bu)$. This is the very
reason that small errors are not magnified in time and allows to obtain uniform in time error bounds.

In this paper we consider a semidiscretization in space with inf-sup stable mixed finite elements for equation \eqref{eq:mod_NS}
 and analyze two different methods:  the Galerkin method and the Galerkin method and grad-div stabilization.
 Grad-div stabilization
was originally proposed in \cite{FH88} to improve the conservation of mass in
finite element methods.  However, it has been observed in the
simulation of turbulent flows, \cite{JK10}, \cite{RL10}, that using only grad-div stabilization
 produced stable (non-oscillating) simulations.
 We prove uniform-in-time error estimates for approximating the unknown reference solution, $\bu$, that corresponds to the coarse spatial scale measurement~$I_H(\bu)$.
 For the Galerkin method without grad-div stabilization, the spatial error bounds we prove are optimal, in the sense that the rate of convergence is that of the best interpolant. In the case we add grad-div stabilization, as in~\cite{grad-div1}, \cite{grad-div2}, we get error bounds in which the error constants do not depend on inverse powers of~the viscosity parameter~$\nu$. This fact is of importance in many applications where viscosity is orders of magnitude smaller than the velocity (i.e., large Rynolds number).  The convergence rates we prove in our error bounds are sharp and
 confirmed by numerical experiments.
% As in~\cite{grad-div1}, \cite{grad-div2}, the error bounds for the stabilized method
% are suboptimal by one order in space. However, as pointed out in \cite{open_problems}, other stabilization procedures should also be added for improving this suboptimal bound, being a suboptimal bound by half an order (instead of one order) the best one that can be found in the literature.

We now comment on the analysis of numerical methods for~\eqref{eq:mod_NS}.  In~\cite{Mondaini_Titi}, a semidiscrete postprocessed Galerkin
spectral method
for the two-dimensional Navier-Stokes equations is studied. Under suitable conditions on the nudging parameter~$\beta$
and the coarse mesh resolution~$H$,
%and the number of modes in the spectral method,
uniform-in-time error estimates are obtained for the difference between the numerical approximation to~$\bv$ and~$\bu$. Furthermore, the use of a postprocessing technique introduced in~\cite{Titi1}~\cite{Titi2}, allows for higher convergence rates than
a standard spectral Galerkin   method. A fully-discrete method for the spatial discretization in~\cite{Mondaini_Titi} is analyzed in~\cite{Ibdah_Mondaini_Titi},
where the backward Euler method is used for time discretization. Fully implicit and semi-implicit methods are considered, and optimal uniform-in-time
error estimates are obtained with the same convergence rate in space as in~\cite{Mondaini_Titi}.
%In~\cite{Mondaini_Titi} and~\cite{Ibdah_Mondaini_Titi}, as opposed to the present paper and~\cite{Larios_et_al},
%they deal only with the two-dimensional case. This, and a different analysis allow the authors of ~\cite{Mondaini_Titi}, \cite{Ibdah_Mondaini_Titi} to obtain error bounds with constants depending only on first spatial derivatives of $\bu$, where the latter is known to be finite.

More closely related to the present work are~\cite{Larios_et_al} and~\cite{Rebholz-Zerfas}.
In~\cite{Rebholz-Zerfas} they only analyze linear problems and, for the proof of the results on the Navier-Stokes equations they
present, they refer to~\cite{Larios_et_al} with some differences that they point out. They also present a wide collection of numerical experiments.
 In~\cite{Larios_et_al},
 the authors consider fully discrete approximations to equation~\eqref{eq:mod_NS}
where the spatial discretization is performed with a MFE Galerkin method plus grad-div stabilization. A second order IMEX in time scheme is analyzed in~\cite{Larios_et_al}, and, as in~\cite{Ibdah_Mondaini_Titi}, \cite{Mondaini_Titi}
 and the present paper, uniform-in-time error bounds are obtained. Compared with~\cite{Larios_et_al}, for the same convergence rate, the error bounds in the present paper
         have constants that do not depend on inverse powers of the viscosity parameter~$\nu$ (Theorem~\ref{Th:main_muno0}) or, for similar error constants, error bounds in the present paper have an order of convergence one unit larger (Theorem~\ref{Th:main} below).
%However,
%as opposed to similar error bounds in the present the present paper  Theorem~\ref{th:main_muno0} ), the error bounds  in~\cite{Larios_et_al}, error constants depend
%on inverse powers of the viscosity parameters, and the order of convergence they prove is one unit less that
%that in~Theorem~\ref{th:main} below.
Also, the analysis in~\cite{Larios_et_al} is restricted to $I_H\bu$ being an
interpolant for non smooth functions (Cl\'ement, Scott-Zhang, etc), since it makes explicit use of bound \eqref{eq:L^2inter}, which is not valid for nodal (Lagrange) interpolation (neither it is~\eqref{eq:cotainter}).
 In the present paper, we prove error bounds for the case in which  \eqref{eq:L^2inter} holds,   but also for the case in which $I_H \bu$ is a standard Lagrange interpolant (Theorem~\ref{Th:main_la} below).
%a case covered in~Theorem~\ref{
%co
%
%suboptimal in space by one order while the error constants
%depend explicitly on inverse powers of the viscosity parameter. Then, compared with~\cite{Larios_et_al} we improve the error bounds till the
%optimal  order for the plain Galerkin method and adapt the techniques in~\cite{grad-div1}, \cite{grad-div2} to get bounds with constants independent on $\nu$ for the  method that adds grad-div stabilization. These are not the only improvements of the present paper compared with previous references. In~\cite{Larios_et_al} the authors state that the operator $I_H \bu$ could be a nodal interpolant but the analysis of
% ~\cite{Larios_et_al} does not work in that case since it makes explicit use of bound \eqref{eq:L^2inter} that does not hold for a Lagrange interpolant. In the present paper, we prove error bounds for the case in which  \eqref{eq:L^2inter} holds,  $I_H \bu$ being
% for example a Scott-Zhang type interpolant, but also for the case in which $I_H \bu$ is a standard Lagrange interpolant.
To our knowledge, this is the first time in the literature where such kind of bounds are proved. Also, compared
 with  \cite{Larios_et_al} and~\cite{Rebholz-Zerfas},  we remove the upper bound assumed on the nudging parameter $\beta$. The authors of  \cite{Larios_et_al} had
 observed (see \cite[Remark 3.8]{Larios_et_al}) that the upper bound they required in the analysis does not hold in the numerical experiments and they state that a different approach to the analysis should be used to remove the upper bound on $\beta$. An analogous upper bound on $\beta$ appears also in \cite{Ibdah_Mondaini_Titi}
 and \cite{Mondaini_Titi},  where the value of $H$ depends on the inverse of the nudging parameter $\beta$ which means that increasing the value of $\beta$ would require a smaller value of $H$.

 Although the analysis of the present paper could be extended to fully discrete methods following for example the techniques in~\cite{grad-div1}, \cite{grad-div2} we believe that the new ideas introduced in the present paper are easier to understand in the framework of the semidiscrete methods. The extension of the analysis of the present paper to the fully discrete case will be subject of future work.

The rest of the paper is as follows. Section~\ref{Se:prelim} is devoted to preliminary material, in Section~\ref{Se:main} we introduce
and analyze the finite element method for equation~\eqref{eq:mod_NS} with and without grad-div stabilization. In Subsection~\ref{sub_la} we analyze the case in which $I_H \bu$ is the standard Lagrange interpolant. Finally, in Section~\ref{se:num} some numerical experiments are shown to illustrate the theoretical results.

\section{Preliminaries and Notation}
\label{Se:prelim}
Throughout the paper, $W^{s,p}(D)$ will denote the Sobolev space of real-valued functions defined on the domain $D\subset\mathbb{R}^d$ with distributional derivatives of order up to $s$ in $L^p(D)$. We denote by~$|\cdot|_{s,p,D}$ standard seminorm, and, following~\cite{Constantin-Foias}, for~$W^{s,p}(D)$ we
will use the norm~$\|\cdot\|_{s,p,D}$ defined by
$$
\left\| f\right\|_{s,p,D}^p=\sum_{j=0}^s \left|D\right|^{\frac{p(j-s)}{d}} \left| f\right|_{j,p,D}^p,
$$
where $|D|$ stands for the Lebesgue measure of~$D$
so that $\left\|f\right\|_{m,p,D}\left|D\right|^{\frac{m}{d}-\frac{1}{p}}$ is scale invariant. If $s$ is not a positive integer, $W^{s,p}(D)$ is defined by interpolation \cite{Adams}.
 In the case $s=0$ one has $W^{0,p}(D)=L^p(D)$. As it is standard, $W^{s,p}(D)^d$ will be endowed with the product norm and, since no confusion can arise, it will be denoted again by $\|\cdot\|_{W^{s,p}(D)}$. The case  $p=2$  will be distinguished by using $H^s(D)$ to denote the space $W^{s,2}(D)$. The space $H_0^1(D)$ is the closure in $H^1(D)$ of the set of infinitely differentiable functions with compact support
in $D$.  For simplicity, $\|\cdot\|_s$ (resp. $|\cdot |_s$) is used to denote the norm (resp. semi norm) both in $H^s(\Omega)$ or $H^s(\Omega)^d$. The exact meaning will be clear by the context. The inner product of $L^2(\Omega)$ or $L^2(\Omega)^d$ will be denoted by $(\cdot,\cdot)$ and the corresponding norm by $\|\cdot\|_0$
 in general $D$ is skipped in the notation for the norm when $D=\Omega$.
For vector-valued functions, the same conventions will be used as before.
The norm of the dual space  $H^{-1}(\Omega)$  of $H^1_0(\Omega)$
is denoted by $\|\cdot\|_{-1}$.
As usual, $L^2(\Omega)$ is always identified
with its dual, so one has $H^1_0(\Omega)\subset L^2(\Omega)\subset H^{-1}(\Omega) $ with compact injection.
The following Sobolev's embedding \cite{Adams} will be used in the analysis: For~$s>0$, let  $1\le p<d/s$
and~$q$ be such that $\frac{1}{q}
= \frac{1}{p}-\frac{s}{d}$. Then, there exists a positive  scale invariant constant $c_s$ such that
%, independent of $s$, such that
\begin{equation}\label{sob1}
\|v\|_{L^{q'}(\Omega)} \le c_s |\Omega|^{\frac{s}{d}-\frac{1}{p}+\frac{1}{q'}}\| v\|_{W^{s,p}(\Omega)}, \qquad
\frac{1}{q'}
\ge \frac{1}{q}, \quad \forall v \in
W^{s,p}(\Omega).
\end{equation}
If $p>d/s$ the above relation is valid for $q'=\infty$.
A similar embedding inequality holds for vector-valued functions.

We will also use the following interpolation inequality~(see, e.g., \cite[formula~(6.7)]{Constantin-Foias} and~\cite[Exercise~II.2.9]{Galdi})
\begin{equation}
\label{eq:parti_ineq}
\left\|v\right\|_{L^{{2d}/{(d-1)}}(\Omega)} \le \textcolor{black}{c_{1}}\left\|v\right\|_0^{1/2}
\left\|  v\right\|_1^{1/2},\qquad \forall v\in H^1(\Omega),
\end{equation}
\textcolor{black}{(where, for simplicity, by enlarging the constants if necessary, we may take the constant~$c_1$ in~(\ref{eq:parti_ineq}) equal to~$c_s$ in~\eqref{sob1} for $s=1$)}
% convexity inequality (see, e.g., \cite[\S~II.1]{Galdi}),
and~Agmon's inequality
\begin{equation}
\label{eq:agmon}
\left\|v \right\|_\infty\le c_{\mathrm{A}} \left\|v\right\|_{d-2}^{1/2} \left\|v\right\|_2^{1/2},\qquad d=2,3,
\qquad \forall v\in H^2(\Omega).
\end{equation}
The case $d=2$ is a direct consequence of~\cite[Theorem 3.9]{Agmon}.
For $d=3$, a proof for domains of class~$C^2$ can be found in~\cite[Lemma~4.10]{Constantin-Foias}. By means of the Calder\'on extension  theorem~(see e.g., \cite[Theorem 4.32]{Adams} the proof is also valid for bounded Lipschitz domains. \textcolor{black}{Finally, we will use Poincar\'e's inequality,
\begin{equation}
\label{Poin}
\left\| v\right\|_0\le c_P|\Omega |^{1/d} \|\nabla v\|_0 ,\qquad \forall v\in H^1_0(\Omega),
\end{equation}
where the constant~$c_P$ can be taken $c_P\le \sqrt{2}/2$.  Denoting
by
\begin{equation}
\label{eq:hatcp}
\hat c_P=1+c_P^2,
\end{equation}
observe that from~(\ref{Poin}) it follows that
\begin{equation}\label{Poin2}
\left\| v\right\|_1\le (\hat c_P)^{1/2} \|\nabla v\|_0,\qquad \forall v\in H^1_0(\Omega).
\end{equation}
}
In all previous inequalities, the constants~$c_s$, $c_1$,
$c_A$ and~$c_P$ are scale-invariant, as it will be the case of all constants in the present paper unless explicitly stated otherwise.

Let $\cal H$ and $V$ be the Hilbert spaces
$
{\cal H}=\{ \bu \in \big(L^{2}(\Om))^d  \, |\, \mbox{div}(\bu)=0, \,
\bu\cdot n_{|_{\partial \Omega}}=0 \}$,
$V=\{ \bu \in \big(H^{1}_{0}(\Om))^d  \, | \, \mbox{div}(\bu)=0 \}$,
endowed with the inner product of $L^{2}(\Om)^{d}$ and $H^{1}_{0}(\Om)^{d},$
respectively.

Let $\mathcal{T}_{h}=(\tau_j^h,\phi_{j}^{h})_{j \in J_{h}}$, $h>0$ be a family of partitions of suitable domains $\Omega_h$, where $h$ is the maximum diameter of the elements $\tau_j^h\in \mathcal{T}_{h}$, and $\phi_j^h$ are the mappings from the reference simplex $\tau_0$ onto $\tau_j^h$.
We shall assume that the partitions are shape-regular and quasi-uniform. Let $r \geq 2$, we consider the finite-element spaces
\begin{eqnarray*}
S_{h,r}&=&\left\{ \chi_{h} \in \mathcal{C}\left(\overline{\Om}_{h}\right) \,  \big|
\, {\chi_{h}}_{|{\tau_{j}^{h}}}
\circ \phi^{h}_{j} \, \in \, P^{r-1}(\tau_{0})  \right\} \subset H^{1}(\Om_{h}),
\nonumber\\
{S}_{h,r}^0&=& S_{h,r}\cap H^{1}_{0}(\Om_{h}),
\end{eqnarray*}
where $P^{r-1}(\tau_{0})$ denotes the space of polynomials of degree at most $r-1$ on $\tau_{0}$. For $r=1$,
$S_{h,1}$ stands for the space of piecewise constants.

When $\Omega$ has polygonal or polyhedral boundary $\Omega_h=\Omega$ and mappings~$\phi_j^h$ from the reference simplex are affine.
When $\Omega$ has a smooth boundary, for the purpose of analysis we will assume that $\Omega_h$ exactly matches $\Omega$, as it is done
for example in~\cite{chenSiam}, \cite{Schatz98}, although at a price of a more complex analysis discrepancies between $\Omega_h$ and
$\Omega$ can also be taken into account (see, e.g., \cite{Ay_Gar_Nov}, \cite{Schatz-Whalbin}).

We shall denote by $(X_{h,r}, Q_{h,r-1})$ the MFE pair known as Hood--Taylor elements \cite{BF,hood0}, when $r\ge 3$, where
\begin{eqnarray*}
X_{h,r}=\left({S}_{h,r}^0\right)^{d},\quad
Q_{h,r-1}=S_{h,r-1}\cap L^2(\Om_{h})/{\mathbb R},\quad r
\ge 3,
\end{eqnarray*}
and, when $r=2$, the MFE pair known as the mini-element~\cite{Brezzi-Fortin91} where $Q_{h,1}=S_{h,2}\cap L^2(\Om_{h})/{\mathbb R}$, and $X_{h,2}=({S}_{h,2}^0)^{d}\oplus{\mathbb B}_h$. Here, ${\mathbb B}_h$ is spanned by the bubble functions $\bbu_\tau$, $\tau\in\mathcal{T}_h$, defined by $\bbu_\tau(x)=(d+1)^{d+1}\lambda_1(x)\cdots
\lambda_{d+1}(x)$,  if~$x\in \tau$ and 0 elsewhere, where $\lambda_1(x),\ldots,\lambda_{d+1}(x)$ denote the barycentric coordinates of~$x$. For these elements a uniform inf-sup condition is satisfied (see \cite{BF}), that is, there exists a constant $\beta_{\rm is}>0$ independent of the mesh grid size $h$ such that
\begin{equation}\label{lbbh}
 \inf_{q_{h}\in Q_{h,r-1}}\sup_{v_{h}\in X_{h,r}}
\frac{(q_{h},\nabla \cdot v_{h})}{\|v_{h}\|_{1}
\|q_{h}\|_{L^2/{\mathbb R}}} \geq \beta_{\rm{is}}.
\end{equation}
The velocity %vector field
will be approximated by elements of~the discrete divergence-free space
\begin{eqnarray*}
V_{h,r}=X_{h,r}\cap \left\{ \chi_{h} \in H^{1}_{0}(\Om_{h})^d \mid
(q_{h}, \nabla\cdot\chi_{h}) =0  \quad\forall q_{h} \in Q_{h,r-1}
\right\}.
\end{eqnarray*}
For each fixed time $t\in[0,T]$ the solution $(u,p)$ of \eqref{NS} is also
the solution of a Stokes problem with right-hand side $\bff-\bu_t-(\bu\cdot\nabla )\bu$. We
will denote by $(\bs_h,q_h)\in(X_{h,r},Q_{h,r-1}),$ its  MFE approximation
satisfying
\begin{eqnarray}
\nu(\nabla \bs_h,\nabla \bvar_h)-(q_h,\nabla \cdot \bvar_h)
&=&\nu(\nabla u,\nabla \bvar_h)
-(p,\nabla\cdot \bvar_h)\nonumber\\
&=&(\bff-\bu_t-(\bu\cdot \nabla \bu),\bvar_h)\quad \forall
\bvar_h\in X_{h,r},
\label{stokesnew}\\
(\nabla \cdot \bs_h,\psi_h)&=&0 \quad \forall \psi_h\in Q_{h,r-1}.\nonumber
\end{eqnarray}
We observe that $\bs_h=S_{h}(\bu) : V \rightarrow V_{h,r}$ is the discrete Stokes
projection of the solution $(\bu,p)$ of \eqref{NS} (see \cite{heyran0}) and satisfies
\begin{eqnarray*}
%\label{defrhns}
 \nu (\nabla S_h(\bu) , \nabla \bvar_{h} )=\nu( \nabla \bu , \nabla \bvar_{h}) -
 (p, \nabla \cdot \bvar_{h})=(\bff-\bu_{t}-(\bu\cdot \nabla )\bu , \bvar_{h}),
 %\quad \forall \, \,
 %\bvar_{h} \in V_{h,r}.
\end{eqnarray*}
for all $\bvar_h\in V_{h,r}$.
The following bound holds:
\begin{equation}
\|\bu-\bs_h\|_0+h\|\bu-\bs_h\|_1\le CN_j(\bu,p) h^j,\qquad
1\le j\le r,
\label{stokespro}
\end{equation}
where here and in the sequel, for $\bv\in V\cap H^j(\Omega)^d$ and~$q\in L^2_0(\Omega)\cap H^{j-1}(\Omega)$ we denote
\begin{equation}
\label{eq:N_j}
 N_j(\bv,q) = \|\bv\|_j + \nu^{-1} \|q\|_{H^{j-1}/{\mathbb R}}, \qquad j\ge 1.
\end{equation}

The proof of \eqref{stokespro} for $\Omega=\Omega_h$ can be found in \cite{heyran2}.
Under the same conditions, the bound for the pressure is (cf.~\cite{girrav})
\begin{equation}\label{stokespre}
\|p-q_h\|_{L^2/{\mathbb R}}\le
C_{\beta_{\rm is}} \nu N_j(\bu,p)h^{j-1}, \qquad 1\le j\le r,
\end{equation}
where the constant $C_{\beta_{\rm is}}$ depends on the constant $\beta_{\rm is}$ in %the inf-sup condition
\eqref{lbbh}.
Assuming that $\Omega$ is of class ${\cal C}^m$,
with $m \ge 3$, and using
standard duality arguments and~\eqref{stokespro}, one obtains
\begin{equation}\label{eq:stokes_menos1}
\| \bu-\bs_{h}\|_{-s}  \leq CN_r(\bu,p) h^{r+s} ,
\qquad 0 \leq s\leq \min(r-2, 1).
\end{equation}
We also consider a modified Stokes projection that was introduced in \cite{grad-div1} and that we denote by $\bs_h^m:V\rightarrow V_{h,r}$
satisfying
\begin{eqnarray}\label{stokespro_mod_def}
\nu(\nabla \bs_h^m,\nabla \bvar_h)=(\bff-\bu_{t}-(\bu\cdot \nabla )\bu-\nabla p , \bvar_{h}), \quad \forall \, \,
 \bvar_{h} \in V_{h,r}.
\end{eqnarray}
The following bound holds, see \cite{grad-div1}:
\begin{equation}
\|\bu-\bs_h^m\|_0+h\|\bu-\bs_h^m\|_1\le C\|\bu\|_j h^j,\qquad
1\le j\le r.
\label{stokespro_mod}
\end{equation}
Following \cite{chenSiam}, one can also obtain the following bound
\begin{align}
\|\nabla (\bu-\bs_h^m)\|_\infty\le C\|\nabla \bu\|_\infty \label{cotainfty1},
\end{align}
where $C$ does not depend on $\nu$.
We will denote by $\pi_h p$ the $L^2$ projection of the pressure $p$ onto $Q_{h,r-1}$. It holds
\begin{equation}\label{eq:L2p}
\|p-\pi_h p\|_0\le C h^{j-1}\|p\|_{H^{j-1}/{\mathbb R}},\qquad 1\le j\le r.
\end{equation}

If the family of
meshes is quasi-uniform then  the following inverse
inequality holds for each $\bv_{h} \in S_{h,r}$, see e.g., \cite[Theorem 3.2.6]{Cia78},
\begin{equation}
\label{inv} \| \bv_{h} \|_{W^{m,p}(K)} \leq c_{\mathrm{inv}}
h_K^{n-m-d\left(\frac{1}{q}-\frac{1}{p}\right)}
\|\bv_{h}\|_{W^{n,q}(K)},
\end{equation}
where $0\leq n \leq m \leq 1$, $1\leq q \leq p \leq \infty$, and $h_K$
%is the size (diameter) of the mesh cell $K \in \mathcal T_h$.
is the diameter of~$K \in \mathcal T_h$.

In the sequel $I_h^{La} \bu\in X_{h,r}$ will denote the Lagrange interpolant of a continuous function $\bu$. The following bound
can be found in~\cite[Theorem 4.4.4]{brenner-scot}
\begin{equation}\label{eq:cotainter_la}
|\bu-I_h^{La} \bu|_{W^{m,p}(K)}\le c_\text{\rm int} h^{n-m}|\bu|_{W^{n,p}(K)},\quad 0\le m\le n\le k+1,
\end{equation}
where $n>d/p$ when $1< p\le \infty$ and $n\ge d$ when $p=1$.

%For the interpolation operator $I_H$ we will assume that it satisfies the following approximation identity property
%We will assume that the interpolation operator $I_H$ satisfies the approximation property
%\begin{eqnarray}\label{eq:cotainter}
%\|\bu-I_H\bu\|_0\le c_I H\|\nabla \bu\|_0,\quad \forall \bu\in H_0^1(\Omega)^d.
%\end{eqnarray}
%We will also assume the $L^2$-stability of the interpolation operator:
%\begin{eqnarray}\label{eq:L^2inter}
%\|I_H \bu\|_0\le c_0\|\bu\|_0,\quad \forall \bu\in L^2(\Omega)^d.
%\end{eqnarray}
We will assume that the interpolation operator $I_H$ is stable in $L^2$, that is,
\begin{eqnarray}\label{eq:L^2inter}
\|I_H \bu\|_0\le c_0\|\bu\|_0,\quad \forall \bu\in L^2(\Omega)^d,
\end{eqnarray}
and that it satisfies the following approximation property,
\begin{eqnarray}\label{eq:cotainter}
\|\bu-I_H\bu\|_0\le c_I H\|\nabla \bu\|_0,\quad \forall \bu\in H_0^1(\Omega)^d.
\end{eqnarray}
The Bernardi--Girault~\cite{Ber_Gir}, Girault--Lions~\cite{Girault-Lions-2001}, or the Scott--Zhang~\cite{Scott-Z} interpolation operators
satisfy
%the above properties in~
\eqref{eq:cotainter} and~\eqref{eq:L^2inter}. Notice that the interpolation can be
on
piecewise constants, as we use in
the numerical experiments in
Section~\ref{se:num}.

%We would like to stress here that we are not assuming for the error analysis condition (3.105) in \cite{Mondaini_Titi}, i.e., we do not assume
%that
%$$
%\|\bu-I_H(\bu)\|_{-1}\le c_{-1} H\|\bu\|_0,\quad \bu\in L^2(\Omega)^d.
%$$
We remark that, for the error analysis, we do not condition (3.105) in \cite{Mondaini_Titi}, i.e., we do not assume
that
$
\|\bu-I_H(\bu)\|_{-1}\le c_{-1} H\|\bu\|_0$,  for $\bu\in L^2(\Omega)^d$.

\section{The finite element method}
\label{Se:main}
We consider the following method to approximate \eqref{eq:mod_NS}. Find $(\bu_h,p_h)\in X_{h,r}\times Q_{h,r-1}$ satisfying
for all $(\bvar_h,\psi_h)\in X_{h,r}\times Q_{h,r-1}$

\begin{align}\label{eq:method}
(\dot \bu_h,\bvar_h)
+\nu(\nabla \bu_h,\nabla \bvar_h)+b_h(\bu_h,&\bu_h,\bvar_h)+\mu(\nabla \cdot \bu_h,\nabla \cdot \bvar_h)
+(\nabla p_h,\bvar_h)\qquad\nonumber\\
&{}=(\bff,\bvar_h)-{\beta(I_H(\bu_h)-I_H(\bu),I_H\bvar_h)},\nonumber\\
(\nabla \cdot \bu_h,\psi_h)&=0,
\end{align}
where $\mu$ is a stabilization parameter that can be zero in case we do not stabilize the divergence or different from zero in case we add grad-div stabilization and $b_h(\cdot,\cdot,\cdot)$  is defined in the following way
$$
b_{h}(\bu_{h},\bv_{h},\bvar_{h}) =((\bu_{h}\cdot \nabla ) \bv_{h}, \bvar_{h})+ \frac{1}{2}( \nabla \cdot (\bu_{h})\bv_{h},\bvar_{h}),
\quad \, \forall \, \bu_{h}, \bv_{h}, \bvar_{h} \in X_{h,r}.
$$
Hereafter, we denote by~$(\cdot,\cdot)$ both the inner product in~$L^2$ and the duality action between~$H^{-1}$ and~$H^{1}_0$, depending
on the context.
It is straightforward to verify that $b_h$ enjoys the skew-symmetry property
\begin{equation}\label{skew}
b_h(\bu,\bv,\bw)=-b_h(\bu,\bw,\bv) \qquad \forall \, \bu, \bv, \bw\in H_0^1(\Omega)^d.
\end{equation}
Let us observe that taking $\bvar_h\in V_{h,r}$ from \eqref{eq:method} we get
\begin{align}\label{eq:method2}
(\dot \bu_h,\bvar_h)
+\nu(\nabla \bu_h,\nabla \bvar_h)+b_h(\bu_h,&\bu_h,\bvar_h)+\mu(\nabla \cdot \bu_h,\nabla \cdot\bvar_h)={}\\
&(\bff,\bvar_h)-\beta(I_H(\bu_h)-I_H(\bu),I_H\bvar_h).\nonumber
\end{align}

%In the sequel we will define
For the analysis below, we need to introduce the values~$\overline\mu$ and
$\overline k$, defined as follows
\begin{equation}
\label{eq:mubarra}
\overline \mu=\left\{\begin{array}{lcl} 0,&\quad& \hbox{\rm if~$\mu=0$},\\ 1,&& \hbox{\rm otherwise},\end{array}\right.
\qquad\qquad
\overline k=\left\{\begin{array}{lcl} 0,&\quad& \hbox{\rm if~$\mu=0$},\\ 1/\mu,&& \hbox{\rm otherwise}.\end{array}\right.
\end{equation}
The following lemma will be used for proving the main results of the section.
%\begin{lema}\label{lema_general}Let $\bw_h:[0,T]\rightarrow V_{h,r}$ be an arbitrary function such that the quantity~$L$
%defined in \eqref{eq:L} when $\mu=0$ and in~\eqref{eq:L_muno0} when $\mu>0$ is bounded. Let $\bu_h$ be the finite element approximation defined in
%\eqref{eq:method2}. Assume that the following equation holds
%\begin{align}\label{eq:wh}
%(\dot \bw_h,\bvar_h)+\nu(\nabla \bw_h,\nabla \bvar_h)+b_h(\bw_h,&\bw_h,\bvar_h)+\mu(\nabla \cdot \bw_h,\nabla \cdot \bvar_h)
%={}\\
%&(\bff,\bvar_h)+(\btau_h^1,\bvar_h)+\overline \mu (\btau_h^2,\nabla \cdot \bvar_h),\nonumber
%\end{align}
%where $\btau_h^1$ and $\btau_h^2$ are truncation errors.
%
%Let us denote by $\be_h=\bu_h-\bw_h$ then the following bound holds
%if $\beta\ge 8L$ %satisfies conditions~\eqref{eq:as1}-\eqref{eq:as1_muno0}
%and $H$ satisfies condition~\eqref{eq:as2} below and $\gamma$ is defined in \eqref{eq:gamma}
%\begin{eqnarray}\label{eq:muboth}
%\|\be_h(t)\|_0^2&\le& e^{-\gamma t/2}\|\be_h(0)\|_0^2+\int_0^t e^{-\gamma(t-s)/2} \left((1-\overline \mu)\frac{2\hat c_P}{\nu}+\frac{{\overline \mu}}{L}\right)\|\btau_h^1\|_{-1+\overline \mu}^2~ds
%\nonumber\\
%&&\quad+\int_0^t e^{-\gamma(t-s)/2}\left( \beta c_0^2\|\bu(s)-\bw_h(s)\|_0^2+ 2{\overline k} \|\tau_h^2\|_0^2\right)~ds,
%\end{eqnarray}
%where
%\begin{equation}
%\label{eq:kbarra}
%\overline k=\left\{\begin{array}{lcl} 0,&\quad& \hbox{\rm if~$\mu=0$},\\ 1/\mu,&& \hbox{\rm otherwise}.\end{array}\right.
%\end{equation}
%\end{lema}
\begin{lema}\label{lema_general}
Let $\bu_h$ be the finite element approximation defined in
\eqref{eq:method2} and let $\bw_h,\btau_h^1,\btau_h^2:[0,T]\rightarrow V_{h,r}$ be functions satisying
\begin{align}\label{eq:wh}
(\dot \bw_h,\bvar_h)+\nu(\nabla \bw_h,\nabla \bvar_h)+b_h(\bw_h,&\bw_h,\bvar_h)+\mu(\nabla \cdot \bw_h,\nabla \cdot \bvar_h)
={}\\
&(\bff,\bvar_h)+(\btau_h^1,\bvar_h)+\overline \mu (\btau_h^2,\nabla \cdot \bvar_h),\nonumber
\end{align}
Assume that the quantity~$L$
defined in \eqref{eq:L}, below, when $\mu=0$, and in~\eqref{eq:L_muno0}, below, when $\mu>0$ is bounded. Then, if $\beta\ge 8L$ and $H$ satisfies condition~\eqref{eq:as2}, below, the following bounds hold for~$\be_h=\bu_h-\bw_h$,
\begin{align}\label{eq:muboth}
\|\be_h(t)\|_0^2\le& e^{-\gamma t/2}\|\be_h(0)\|_0^2+\int_0^t e^{-\gamma(t-s)/2} \left((1-\overline \mu)\frac{2\hat c_P}{\nu}+\frac{{\overline \mu}}{L}\right)\|\btau_h^1\|_{-1+\overline \mu}^2~ds
\nonumber\\
&{}+\int_0^t e^{-\gamma(t-s)/2}\left( \beta c_0^2\|\bu(s)-\bw_h(s)\|_0^2+ 2{\overline k} \|\tau_h^2\|_0^2\right)~ds,
\end{align}
where, $\overline\mu$ and~$\overline k$ are defined in~\eqref{eq:mubarra}, and  $\gamma$ is defined in \eqref{eq:gamma} below.
%\begin{equation}
%\label{eq:kbarra}
%\overline k=\left\{\begin{array}{lcl} 0,&\quad& \hbox{\rm if~$\mu=0$},\\ 1/\mu,&& \hbox{\rm otherwise}.\end{array}\right.
%\end{equation}
\end{lema}
\begin{proof} Subtracting \eqref{eq:wh} from \eqref{eq:method2} we get the error equation
\begin{align}\label{eq:error1}
&(\dot\be_h,\bvar_h)+\nu(\nabla\be_h,\nabla\bvar_h)+\beta(I_H \be_h,I_H\bvar_h)+b_h(\bu_h,\bu_h,\bvar_h)
-b_h(\bw_h,\bw_h,\bvar_h)\nonumber\\
&{}+\mu(\nabla \cdot \be_h,\nabla \cdot \bvar_h)=
\beta(I_H \bu-I_H \bw_h,I_H \bvar_h)+(\btau_h^1,\bvar_h)+\overline \mu(\btau_h^2,\nabla \cdot \bvar_h),
\end{align}
for all $\bvar_h\in V_{h,r}$
Taking $\bvar_h=\be_h$ in \eqref{eq:error1} we get
\begin{align}\label{eq:error2}
&\frac{1}{2}\frac{d}{dt}\|\be_h\|_0^2+\nu\|\nabla \be_h\|_0^2+\beta \|I_H\be_h\|_0^2+\mu\|\nabla \cdot \be_h\|_0^2\le |b_h(\bu_h,\bu_h,\be_h)
\\
&{}-b_h(\bw_h,\bw_h,\be_h)|
 +\beta|(I_H \bu-I_H \bw_h,I_H\be_h)|+|(\btau_h^1,\be_h)|+|\overline \mu(\btau_h^2,\nabla \cdot \be_h)|.\quad\nonumber
\end{align}
We will bound the terms on the right-hand side of \eqref{eq:error2}. For the nonlinear term and the truncation errors
%we use a different argument depending on the case $\mu=0$ or $\mu>0$.
we argue differently  depending on whether $\mu=0$ or $\mu>0$.

If $\mu=0$, using the
skew-symmetry property \eqref{skew},  \textcolor{black}{\eqref{sob1} and~(\ref{Poin2})}, and when $d=3$, we have
\begin{align}\label{eq:error3}
|b_h(\bu_h,\bu_h,\be_h)
-b_h(&\bw_h,\bw_h,\be_h)|=|b_h(\be_h,\bw_h,\be_h)|
\\
{}\le& \|\nabla \bw_h\|_{L^{2d/(d-1)}}\|\be_h\|_{L^{2d}}\|\be_h\|_0
+\frac{1}{2}|(\nabla \cdot \be_h)\bw_h,\be_h)|\nonumber\\
{}\le &\textcolor{black}{\hat c_1}\|\nabla \bw_h\|_{L^{2d/(d-1)}}\|\nabla \be_h\|_0\|\be_h\|_0
+\frac{1}{2}\|\nabla \be_h\|_0\|\bw_h\|_\infty\|\be_h\|_0\nonumber\\
{}\le& \left(2\textcolor{black}{\hat c_1^2}\frac{\|\nabla\bw_h\|_{L^{2d/(d-1)}}^2}{\nu}+\frac{\|\bw_h\|_\infty^2}{\nu}
\right)\|\be_h\|_0^2+\frac{\nu}{4}\|\nabla \be_h\|_0^2,
\nonumber
\end{align}
\textcolor{black}{where
\begin{equation}
\label{eq:hatc1}
\hat c_1=(\hat c_P)^{1/2}c_1,
\end{equation}
 $c_1$ being the constant in~\eqref{sob1} for~$s=1$}.
% \todo{hemos aplicado \eqref{sob1} con $\|\nabla \be_h\|_0$ en lugar
%de $\|\be_h\|_1$ tal vez se pueda comentar o a\~nadir una constante num\'erica si falta}  DONE.
In the case $d=2$, and noticing that $2d=2d/(d-1)$, the first term on the right-hand side above, using~\eqref{eq:parti_ineq}, (\ref{Poin2}) and Young's inequality is bounded as follows
\begin{align}
 \|\nabla \bw_h\|_{L^{\frac{2d}{d-1}}}\|\be_h\|_{L^{2d}}\|\be_h\|_0 &\le  \textcolor{black}{(\hat c_P)^{1/4}}c_1\|\nabla \bw_h\|_{L^{\frac{2d}{d-1}}}(\|\nabla \be_h\|_{0}\|\be_h\|_0)^{1/2}
 \|\be_h\|_0\nonumber\\
 &\le \textcolor{black}{{3}(\hat c_P)^{1/3}}c_1^{4/3}\frac{\|\nabla\bw_h\|_{L^{\frac{2d}{d-1}}}^{4/3}}{(\textcolor{black}{4}\nu)^{1/3}}\|\be_h\|_0^2 +
  \frac{\nu}{\textcolor{black}{4}}\|\nabla \be_h\|_0^2.
  \label{eq:engorro1}
\end{align}
For the truncation error when $\mu=0$ using \eqref{Poin2} we get
\begin{equation}\label{eq:error6}
|(\btau_h^1,\be_h)|\le \|\btau_h^1\|_{-1}\|\be_h\|_1\le \textcolor{black}{(\hat c_P)^{1/2}}\|\btau_h^1\|_{-1}\|\nabla \be_h\|_0
\le \frac{\textcolor{black}{\hat c_P}}{\nu}\|\btau_h^1\|_{-1}^2+\frac{\nu}{4}\|\nabla \be_h\|_0^2.
\end{equation}

When $\mu\neq 0$, we bound the nonlinear term in the following way. Using again the
skew-symmetry property \eqref{skew} we get
\begin{align}\label{eq:error3_muno0}
|&b_h(\bu_h,\bu_h,\be_h)
-b_h(\bw_h,\bw_h,\be_h)|=|b_h(\be_h,\bw_h,\be_h)|\le \|\nabla \bw_h\|_{\infty}\|\be_h\|_0^2\\
&{}+\frac{1}{2}\|\nabla \cdot \be_h\|_0\|\bw_h\|_\infty\|\be_h\|_0\le \|\nabla \bw_h\|_{\infty}\|\be_h\|_0^2
+\frac{\mu}{4}\|\nabla \cdot \be_h\|_0^2+\frac{\|\bw_h\|_\infty^2}{4\mu}\|\be_h\|_0^2.
\nonumber
\end{align}
In the sequel we denote
\begin{eqnarray}\label{eq:L}
L&=&\max_{t\ge0}\left(2\frac{\hat c_1^{2d/3}\|\nabla\bw_h(t)\|_{L^{2d/(d-1)}}^{2d/3}}{\nu^{(2d-3)/3}}+\frac{\|\bw_h(t)\|_\infty^2}{\nu}\right),\quad {\rm if}\quad \mu=0,
\\
\label{eq:L_muno0}
L&=&2\max_{t\ge 0}\left(\|\nabla \bw_h(t)\|_\infty+ \frac{\|\bw_h(t)\|_\infty^2}{4\mu}\right),\hfill\quad {\rm if}\quad \mu> 0,
\end{eqnarray}
Observe that in the case~$\mu=0$, bounding the factor $3(\hat c_P)^{1/3}/4^{1/3} $ in~(\ref{eq:engorro1}) by~$2(\hat c_P)^{2/3}$ we have the left-hand side of~(\ref{eq:error3}) can be bounded by $L\|\be_h\|_0^2+(\nu/2)\|\nabla\be_h\|_0^2$, and, in the case~$\mu>0$ the left-hand side of~(\ref{eq:error3_muno0}) is bounded by~$(L/2)\|\be_h\|_0^2
+(\mu/4)\|\nabla\cdot\be_h\|_0^2$.

Next, we bound the truncation error when $\mu>0$,
\begin{align}\label{eq:error6_muno0}
|(\btau_h^1,\be_h)|+|\overline \mu(\btau_h^2,\nabla \cdot \be_h)|\le \frac{1}{2L}\|\btau_h^1\|_0^2+\frac{L}{2}\|\be_h\|_0^2
+\overline k \|\tau_h^2\|_0^2+ \frac{\mu}{4}\|\nabla \cdot \be_h\|_0^2,
\end{align}
where $\overline k$ is defined in \eqref{eq:mubarra}.

For the second term on the right-hand side of \eqref{eq:error2} applying \eqref{eq:L^2inter} we get
\begin{eqnarray}\label{eq:error5}
\beta|(I_H \bu-I_H \bw_h,I_H\be_h)|&\le &\beta c_0\|\bu-\bw_h\|_0\|I_H\be_h\|_0
\nonumber \\
&\le& \frac{\beta}{2} c_0^2\|\bu-\bw_h\|_0^2
+\frac{\beta}{2}\|I_H\be_h\|_0^2.
\end{eqnarray}
Inserting \eqref{eq:error3}, \eqref{eq:engorro1}, \eqref{eq:error6}, \eqref{eq:error3_muno0}, \eqref{eq:error6_muno0} and \eqref{eq:error5} into \eqref{eq:error2} we get
\begin{align}\label{eq:otra}
\frac{1}{2}\frac{d}{dt}\|\be_h\|_0^2+&(1+\overline  \mu)\frac{\nu}{2}\|\nabla \be_h\|_0^2+\frac{\beta}{2} \|I_H \be_h\|_0^2
+\frac{\mu}{2}\|\nabla \cdot \be_h\|_0^2
\le L\|\be_h\|_0^2
\\ &{}+\overline k \|\tau_h^2\|_0^2+\frac{\beta}{2} c_0^2\|\bu-\bw_h\|_0^2+\left((1-\overline \mu)\frac{\hat c_P}{\nu}+\frac{\overline \mu}{2L}\right)\|\btau_h^1\|_{-1+\overline \mu}^2.
\nonumber
\end{align}
Now we bound
\begin{eqnarray*}
L\|\be_h\|_0^2\le 2L \|I_H e_h\|_0^2+2L \|(I-I_H)e_h\|_0^2.
\end{eqnarray*}
Since we are assuming that $\beta\ge 8L$ we have that $\beta/2-2L\ge \beta/4$, so that taking into account that $1+\overline \mu\ge 1$ and~$(\mu/2)\|\nabla\cdot\be_h\|\ge 0$ we get
\begin{align}
\label{eq:aver}
\frac{d}{dt}\|\be_h\|_0^2+{\nu}\|\nabla &\be_h\|_0^2
-4L\|(I-I_H)\be_h\|_0^2+\frac{\beta}{2} \|I_H \be_h\|_0^2
\le {}
\\ &2\overline k \|\tau_h^2\|_0^2+\beta c_0^2\|\bu-\bw_h\|_0^2+\left((1-\overline \mu)\frac{2\hat c_P}{\nu}+\frac{{\overline \mu}}{L}\right)\|\btau_h^1\|_{-1+\overline \mu}^2.
\nonumber
\end{align}
For the second and third terms on the left-hand side above, applying \eqref{eq:cotainter} to the latter, we write
\begin{eqnarray}
\label{eq:apply(21)a}
{\nu}\|\nabla \be_h\|_0^2-4L\|(I-I_H)e_h\|_0^2\ge {\nu}\|\nabla \be_h\|_0^2-4 L c_I^2 H^2 \|\nabla e_h\|_0^2\ge \frac{\nu}2\|\nabla e_h\|_0^2,
\end{eqnarray}
whenever
\begin{equation}
\label{eq:as2}
H\le \frac{\nu^{1/2}}{(8L)^{1/2}c_I}.
\end{equation}
Therefore, for the last three terms on the left-hand side of~(\ref{eq:aver}) we have
\begin{equation}\label{eq:aver2}
{\nu}\|\nabla \be_h\|_0^2+\frac{\beta}{2} \|I_H \be_h\|_0^2
-4L\|(I-I_H)\be_h\|_0^2\ge \frac{\nu}{2}\|\nabla \be_h\|_0^2+\frac{\beta}{2} \|I_H \be_h\|_0^2.
\end{equation}
Now, applying \eqref{eq:cotainter} again to bound below  the right-hand side above we have that
\begin{eqnarray}\label{eq:apply(21)b}
 \frac{\nu}{2}\|\nabla \be_h\|_0^2+\frac{\beta}{2} \|I_H \be_h\|_0^2 &\ge &\frac{\nu}{2}c_I^{-2}H^{-2}\|(I-I_H)e_h\|_0^2
 +\frac{\beta}{2} \|I_H \be_h\|_0^2\\
 {}&\ge & {\gamma} (\|I_H\be_h\|_0^2+\|(I-I_H)\be_h\|_0^2),
 \nonumber
\end{eqnarray}
where
\begin{equation}
\label{eq:gamma}
\gamma=\min\left\{\frac{\nu}{2}c_I^{-2}H^{-2},\frac{\beta}{2}\right\}.
\end{equation}
Finally, since
$
 {\gamma} (\|I_H\be_h\|_0^2+\|(I-I_H)\be_h\|_0^2)\ge (\gamma/2) \|e_h\|_0^2
$,
from~\eqref{eq:aver}, \eqref{eq:aver2} and~\eqref{eq:apply(21)b} it follows that
\begin{eqnarray*}
\frac{d}{dt}\|\be_h\|_0^2+\frac{\gamma}{2} \|\be_h\|_0^2
\le 2\overline k \|\tau_h^2\|_0^2+\beta c_0^2\|\bu-\bw_h\|_0^2
 %\nonumber \\
 +\left(\!\!(1-\overline \mu)\frac{2\hat c_P}{\nu}+\frac{{\overline \mu}}{L}\right)\!\|\btau_h^1\|_{-1+\overline \mu}^2,
\end{eqnarray*}
from which we reach \eqref{eq:muboth}.
\end{proof}
%In the following theorem we prove the rate of convergence of the unstabilized Galerkin method (case $\mu=0$).
We now obtain the error bounds of the standard Galerkin method (case $\mu=0$).
\begin{Theorem}\label{Th:main}
Assume that the solution of~\eqref{NS} satisfies that $\bu\in L^\infty(H^s(\Omega)^d)$,
$p\in L^\infty (H^{s-1}(\Omega)/{\mathbb R})$, $\bu_t\in  L^\infty(H^{\max(2,s-1)}(\Omega)^d)$ and~$p_t\in L^\infty( H^{\max(1,s-2)}(\Omega)
/{\mathbb R})$ for $s\ge 2$. Let $\bu_h$ be the finite element approximation defined in
\eqref{eq:method2} with $\mu=0$.
%Then, if $\beta$ satisfies condition~\eqref{eq:as1_u} and $H$ satisfies condition~\eqref{eq:as2}  the following bound holds for
Then, if $\beta\ge 8L$ and $H$ satisfies condition~\eqref{eq:as2}  the following bound holds for $t\ge 0$ and $2\le r\le s$,
\begin{align*}
\| \bu(t)-\bu_h(t)\|_0 \le  &e^{-\gamma t/2}\|\bu_h(0)-\bu(0)\|_0^2\\
&{}+C\Bigl(\max_{0\le \tau\le t} \bigl((\beta/\gamma)^{1/2}+(\gamma\nu)^{-1/2}K_0(\bu,p,|\Omega|)\bigr)
N_r(\bu,p)
\\
&{}+
(\gamma\nu)^{-1/2} \max_{0\le \tau\le t} |\Omega|^{(1+\hat r-r)/d}N_{\hat r}(\bu_t,p_t)\Bigr)h^r,
\end{align*}
where $\gamma$ is defined in \eqref{eq:gamma}, $K_0(\bu,p,|\Omega|)$ is defined in~\eqref{Cup} below and $\hat r=r-1$ if $r\ge 3$ and $\Omega$ is of class~${\cal C}^3$ and $\hat r=r$ otherwise.
\end{Theorem}

\begin{proof}
Following \cite{Ay_Gar_Nov} we compare $\bu_h$ with $\bs_h$, where~$\bs_h$ satisfies~\eqref{stokesnew} for which we apply Lemma~\ref{lema_general}
with $\bw_h=\bs_h$. To bound
% the norms
 $\|\bs_h\|_\infty$ and $\|\nabla \bs_h\|_{L^{2d/(d-1)}}$ in \eqref{eq:L} we apply \eqref{cota_sh_inf} and \eqref{la_cota}.

We observe that equation \eqref{eq:wh} holds with $\mu=0$ and $\btau_h^2=0$ and
$$
(\btau_h^1,\bvar_h)=(\bu_t-\dot \bs_h,\bvar_h)+b_h(\bu,\bu,\bvar_h)-b_h(\bs_h,\bs_h,\bvar_h),\quad \forall \bvar_h\in V_{h,r}.
$$
Then from \eqref{eq:muboth} we get
\begin{eqnarray*}
\|\be_h(t)\|_0^2&\le& e^{-\gamma t/2}\|\be_h(0)\|_0^2+\int_0^t e^{-\gamma(t-s)/2} \frac{2\hat c_P}{\nu}\|\btau_h^1\|_{-1}^2~ds
\nonumber\\
&&\quad+\int_0^t e^{-\gamma(t-s)/2} \beta c_0^2\|\bu(s)-\bw_h(s)\|_0^2~ds.
\end{eqnarray*}
Consequently,
\begin{align*}%\label{eq:final}
\|\be_h(t)\|_0^2\le& e^{-\gamma t/2}\|\be_h(0)\|_0^2+\frac{4\textcolor{black}{\hat c_P}}{\nu \gamma}\max_{0\le \tau\le t}\|\tau_h(\tau)\|_{-1}^2
%\nonumber\\ &{}
+2 c_0^2\frac{\beta}{\gamma}\max_{0\le \tau\le t}\|\bu(\tau)-\bs_h(\tau)\|_0^2.
\end{align*}
To bound the last term on the right-hand side of above we apply \eqref{stokespro} to get
$$
\max_{0\le \tau\le t}\|\bu(\tau)-\bs_h(\tau)\|_0^2\le C h^{2r}\max_{0\le \tau\le t}N_r(\bu(\tau),p(\tau)).
$$
For the truncation error, applying \eqref{eq:stokes_menos1} we can bound
$$
\|\bu_t-\dot \bs_h\|_{-1} \le C h^rN_{r-1}( \bu_t,p_t),
$$
or, in case we use the mini-element or the boundary is not of class~${\cal C}^3$, applying \eqref{stokespro} again we get
$$
\| \bu_t-\dot \bs_h\|_{-1} \le C|\Omega|^{1/d} h^rN_{r}( \bu_t,p_t).
$$
Also, applying~Lemma~\ref{le:est-1} below
we have
$$
\sup_{\|\bvar\|_1=1}|b_h(\bu,\bu,\bvar)-b_h(\bs_h,\bs_h,\bvar)|\le K_0 (\bu,p,|\Omega|)\|\bu-\bs_h\|_0,
$$
so that we conclude the proof by applying again \eqref{stokespro}.
\end{proof}

We observe from Theorem~\ref{Th:main} that the rate of convergence of the method is optimal $O(h^s)$ and, as in \cite{Mondaini_Titi}, we have obtained uniform
in time error estimates. In the following theorem we bound the error of the Galerkin method with grad-div stabilization (case $\mu >0$). Comparing with Theorem~\ref{Th:main}
we show that adding grad-div stabilization allows to remove the dependence of the error constants on inverse powers of the viscosity $\nu$.
\begin{Theorem}\label{Th:main_muno0}
Assume that  the solution of~\eqref{NS} satisfies that $\bu\in L^\infty(H^s(\Omega)^d)\cap W^{1,\infty}(\Omega)^d$,
$p\in L^\infty (H^{s-1}(\Omega)/{\mathbb R})$, $\bu_t\in  L^\infty(H^{s-1}(\Omega)^d)$ for $s\ge 2$. Let $\bu_h$ be the finite element approximation defined in
\eqref{eq:method2} with grad-div stabilization ($\mu\neq 0$).
%Then, if $\beta$ satisfies condition~\eqref{eq:as1_muno0_u} and $H$ satisfies condition~\eqref{eq:as2}
Then, if $\beta\ge 8L$ and $H$ satisfies condition~\eqref{eq:as2}
the following bound holds for $t\ge 0$ and $2\le r\le s$,
\begin{align*}
\| \bu(t)-\bu_h(t)\|_0 \le& e^{-\gamma t/2}\|\bu_h(0)-\bu(0)\|_0^2\\
&{}+\frac{C}{L^{1/2}}h^{r-1}\!\!\max_{0\le \tau\le t} \biggl(\Bigl(\beta^{1/2}h+\mu^{1/2}+\frac{K_1(\bu,|\Omega|)}{L^{1/2}}
\Bigr)\|\bu\|_r  \\
&{} +\frac{1}{L^{1/2}} \|\bu_t\|_{r-1}+\frac{1}{\mu^{1/2}}\|p\|_{H^{r-1}/{\mathbb R}})\biggr),
\end{align*}
where $\gamma$ is defined in \eqref{eq:gamma} and $K_1(\bu,|\Omega|)$ is defined in~\eqref{Cu} below.
\end{Theorem}
\begin{proof} Following \cite{grad-div1}, \cite{grad-div2} we compare $\bu_h$ with $\bs_h^m$, where~$\bs_h^m$ satisfies~\eqref{stokespro_mod_def}.
We first observe that the norms   in \eqref{eq:L_muno0} are bounded since for $\|\bs_h^m\|_\infty$ we apply  \eqref{cota_sh_inf_mu} and
applying
\eqref{cotainfty1} $\|\nabla \bs_h^m\|_\infty\le C \|\nabla \bu\|_\infty$.

Then, we apply Lemma~\ref{lema_general}
with $\bw_h=\bs_h^m$. We observe that %equation
\eqref{eq:wh} holds with
$$
(\btau_h^1,\bvar_h)=(\dot\bu-\dot \bs_h^m,\bvar_h)+b_h(\bu,\bu,\bvar_h)-b_h(\bs_h^m,\bs_h^m,\bvar_h),\quad \forall \bvar_h\in V_{h,r},
$$
and
$$
(\btau_h^2,\nabla \cdot\bvar_h)=(\pi_h p-p,\nabla \cdot \bvar_h)+\mu(\nabla \cdot (\bu-\bs_h^m),\nabla \cdot \bvar_h)
$$
and then  from \eqref{eq:muboth} we get
\begin{eqnarray*}
\|\be_h(t)\|_0^2&\le& e^{-\gamma t/2}\|\be_h(0)\|_0^2+\int_0^t e^{-\gamma(t-s)/2}\frac{ \|\btau_h^1\|_{0}^2}{L}~ds
\nonumber\\
&&\quad+\int_0^t e^{-\gamma(t-s)/2}\left( \beta c_0^2\|\bu(s)-\bs_h^m(s)\|_0^2+\frac{ 2}{\mu} \|\tau_h^2\|_0^2\right)~ds.
\end{eqnarray*}
Consequently,
\begin{align*}
\|\be_h(t)\|_0^2\le e^{-\gamma t/2}\|\be_h(0)\|_0^2+\frac{2}{\gamma L}\max_{0\le \tau\le t}\|\tau_h^1(\tau)\|_{0}^2
%\nonumber\\ &{}
+&2 c_0^2\frac{\beta}{\gamma}\max_{0\le \tau\le t}\|\bu(\tau)-\bs_h^m(\tau)\|_0^2\nonumber\\
&+\frac{4}{\mu\gamma}\max_{0\le \tau\le t}\|\btau_h^2(\tau)\|_0^2.
\end{align*}
From \eqref{eq:as2} and \eqref{eq:gamma} and taking into account that we are assuming $\beta\ge 8L$ we get
$
{1}/{\gamma}\le \max\left({2}/{\beta},{1}/{(4L)}\right)={1}/{(4L)}$ and ${\beta}/{\gamma}\le \max\left(2,{\beta}/{(4L)}\right)
={\beta}/{(4L)}.
$
Then,  it follows that
\begin{align}\label{eq:final_muno0}
\|\be_h(t)\|_0^2\le e^{-\gamma t/2}\|\be_h(0)\|_0^2+\frac{1}{2L^2}&\max_{0\le \tau\le t}\|\tau_h^1(\tau)\|_{0}^2
%\nonumber\\ &{}
+\frac{\beta}{2L} c_0^2\max_{0\le \tau\le t}\|\bu(\tau)-\bs_h^m(\tau)\|_0^2\nonumber\\
&+\frac{1}{\mu L}\max_{0\le \tau\le t}\|\btau_h^2(\tau)\|_0^2.
\end{align}
To bound the second term on the right-hand side of \eqref{eq:final_muno0} we apply \eqref{stokespro_mod} to get
$$
\max_{0\le \tau\le t}\|\bu(\tau)-\bs_h(\tau)\|_0^2\le C h^{2r}\max_{0\le  \tau\le t}\|\bu(\tau)\|_r^2.
$$
For the first term in the truncation error $\btau_h^1$ we apply \eqref{stokespro_mod} again to get
$$
\max_{0\le \tau\le t}\|\bu_t(\tau)- \dot\bs_h(\tau)\|_0^2\le C h^{2(r-1)}\max_{0\le \tau\le t}\| \bu_t( \tau)\|_{r-1}^2.
$$
For the second term in the truncation error $\btau_h^1$, applying~Lemma~\ref{le:est-1} below
we have
\begin{align*}
\sup_{\|\bvar\|_0=1}|b_h(\bu,\bu,\bvar)-b_h(\bs_h^m,\bs_h^m,\bvar)|&\le K_1(\bu,|\Omega|)\|\bu-\bs^m_h\|_1
\\
&\le CK_1(\bu,|\Omega|)h^{r-1}\|\bu\|_{r},
\end{align*}
%\end{eqnarray*}
where in the last inequality we have applied \eqref{stokespro_mod}.
Finally, from \eqref{eq:L2p} and \eqref{stokespro_mod} we obtain
\begin{eqnarray*}
\|\btau_h^2\|_0\le C h^{r-1}\|p\|_{H^{r-1}/{\mathbb R}}+C\mu h^{r-1}\|\bu\|_r,
\end{eqnarray*}
which concludes the proof.
\end{proof}
%\begin{remark}\label{re:proy} Some works in the literature \cite{Larios_et_al}, \cite{Mondaini_Titi}, use $\beta(I_H (\bu_h-\bu),\bvar_h)$ as nudging term
%instead of the one in~\eqref{eq:method}, which is also used in~\cite{Rebholz-Zerfas}. In the case where $I_H$ is the orthogonal projection in~$L^2$, since $\beta(I_H (\bu_h-\bu),\bvar_h)=\beta(I_H (\bu_h-\bu),I_H \bvar_h)$ the analysis presented above  covers obviously the case where the nudging term is
%$\beta(I_H (\bu_h-\bu),\bvar_h)$.
%\end{remark}
\begin{remark}\label{re:proy}\rm Some works in the literature \cite{Larios_et_al}, \cite{Mondaini_Titi}, use $\beta(I_H (\bu_h-\bu),\bvar_h)$ as nudging term
instead of the one in~\eqref{eq:method}, which is also used in~\cite{Rebholz-Zerfas}. In the case where $I_H$ is the orthogonal projection in~$L^2$, since $\beta(I_H (\bu_h-\bu),\bvar_h)=\beta(I_H (\bu_h-\bu),I_H \bvar_h)$ the analysis presented above  obviously covers both nudging terms.
\end{remark}

\begin{remark}\label{re:nu_large}\rm By adding ${}+\mu(\nabla\cdot\bs_h,\nabla\cdot
\bvar)$ to the left hand side of the first equation in~\eqref{stokesnew}, and repeating
the arguments in the proof of Theorem~\ref{Th:main} (with obvious changes), one
can obtain an $O(h^s)$ error bound also when~$\mu>0$, but where,
as in~Theorem~\ref{Th:main} and opposed to~Theorem~\ref{Th:main_muno0}, error constants depend on inverse powers of~$\nu$ and, hence, are useful in practice only when $\nu$ is not too small (see Fig.~\ref{fig1} below).
\end{remark}
\begin{lema}\label{le:est-1} The following bounds hold
\begin{eqnarray}
\label{eq:K_0}
\sup_{\|\bvar\|_1=1}| b_h(\bu,\bu,\bvar)-b_h(\bs_h,\bs_h,\bvar)| &\le& K_0(\bu,p,|\Omega|) \|\bu-\bs_h\|_0,
\\
\sup_{\|\bvar\|_0=0}| b_h(\bu,\bu,\bvar)-b_h(\bs_h^m,\bs_h^m,\bvar)| &\le& K_1(\bu,|\Omega|) \|\bu-\bs_h^m\|_1,
\label{eq:K_1}
\end{eqnarray}
where
\begin{equation}
\label{Cup}
K_0(\bu,p,|\Omega|)\!=\!C\Bigl( K_1(\bu,|\Omega|)
%&{}
+N_1(\bu,p)^{1/2}\bigl(N_{d-1}(\bu,p)+
|\Omega|^{(3-d)/d}N_{2}(\bu,p)\bigr)^{1/2}\Bigr),
\end{equation}
\begin{equation}
\label{Cu}
K_1(\bu,|\Omega|) \!=\!C\Bigl((\|\bu\|_{d-2}\|\bu\|_2)^{1/2} + |\Omega|^{(3-d)/(2d)}(\|\bu\|_1\|\bu\|_2)^{1/2}\Bigl),
\end{equation}
and $N_j(\bu,p)$ is the quantity in~\eqref{eq:N_j}.
\end{lema}
\begin{proof}
Applying~\cite[Lemma~5]{proyNS} we have
\begin{align*}
| b_h(\bu,\bu,\bvar)-b_h(\bs_h,\bs_h,&\bvar)| \\
{}\le&  C\bigl(\|\nabla \bu\|_{L^{2d/(d-1)}} + \|\nabla \bs_h\|_{L^{2d/(d-1)}}\bigr)\|\bu -\bs_h\|_0\|\bvar\|_{L^{2d}}
\nonumber\\
&{}+\bigl(\|\bu\|_\infty+\|\bs_h\|_\infty\bigr)\|\bu -\bs_h\|_0\|\nabla \bvar\|_0.
\end{align*}
To bound~$\|\nabla\bu\|_{L^{2d/(d-1)}}$ and~$\|\bu\|_\infty$ we apply \eqref{eq:parti_ineq} and~\eqref{eq:agmon}, respectively,
and applying~Sobolev's inequality~\eqref{sob1} we have~$\|\bvar\|_{L^{2d}}\le c_1|\Omega|^{(3-d)/(2d)}\|\bvar\|_1$.
The proof of~(\ref{eq:K_0}) is finished by applying Lemma~\ref{le:est_sh_inf} below.

To prove~(\ref{eq:K_1}), we replace~$\bs_h$ by~$\bs_h^m$ in the arguments above, and use the skew-symmetric property of~$b$ to interchange the roles of~$\varphi$ and~$\bu -
\bs_h$. We finish by applying~Lemma~\ref{le:est_sh_inf_mu} below.
\end{proof}
%We now bound $\|\bs_h\|_\infty$ and~$\|\nabla\bs_h\|_{L^{2d/(d-1)}}$.
\begin{lema}\label{le:est_sh_inf} There exist a positive constant~$C_0$ such that the following bounds hold
\begin{align}
\label{cota_sh_inf}
\|\bs_h\|_\infty  & \le C_0\left((\|\bu\|_{d-2}\|\bu\|_2)^{1/2} +\bigl(N_1(\bu,p)N_{d-1}(\bu,p)\bigr)^{1/2}\right)
\\
\label{la_cota}
\|\nabla\bs_h\|_{L^{2d/(d-1)}} & \le
 C_0\bigl(N_1(\bu,p)N_2(\bu,p)\bigr)^{1/2}
\end{align}
\end{lema}
\begin{proof}
For the $L^\infty$ bound, applying inverse inequality \eqref{inv}, we write
\begin{align*}
\|\bs_h\|_\infty &\le C\| I_h(\bu)\|_\infty + \|\bs_h-I_h(\bu)\|_\infty\\
&\le C\|\bu\|_\infty + c_{\rm inv} h^{-d/2} \bigl( \|\bs_h-\bu\|_0+\|\bu -I_h(\bu)\|_0\bigr),
\end{align*}
an apply \eqref{eq:agmon} to bound $\|\bu\|_\infty$.
In the case $d=2$ we have
$$
\|\bu -I_h(\bu)\|_0 \le Ch \|\bu\|_1\le Ch (\|\bu\|_0\|\bu\|_2)^{1/2} = Ch^{d/2}  (\|\bu\|_{d-2}\|\bu\|_2)^{1/2},
$$
where we have applied \eqref{eq:cotainter} for $H=h$ and also $\|\bu -I_h(\bu)\|_0 \le C h^2\|\bu\|_2$.
By \eqref{stokespro}
$$
\|\bs_h-\bu\|_0\le  CN_1(\bu,p)h.
$$
In the case $d=3$,
$$
\|\bu -I_h(\bu)\|_0 \le Ch^{3/2}   (\|\bu\|_1\|\bu\|_2)^{1/2} =Ch^{d/2}  (\|\bu\|_{d-2}\|\bu\|_2)^{1/2}.
$$
and
\begin{align*}
\|\bs_h-\bu\|_0 &%\nonumber\\
\le Ch^{d/2} (N_1(\bu,p)N_2(\bu,p))^{1/2}.
\end{align*}
For~$\|\nabla\bs_h\|_{L^{2d/(d-1)}}$, since $\|\nabla \bs_h\|_{L^q} \le C(\|\nabla \bu\|_{L^q}+\nu^{-1}\|p\|_{L^q})$, for $q=2,\infty$
\cite{chenSiam}, by the Riesz-Thorin interpolation theorem and applying \eqref{eq:parti_ineq}
\begin{align*}
\|\nabla\bs_h\|_{L^{2d/(d-1)}} &\le C(\|\nabla \bu\|_{L^{2d/(d-1)}}+\nu^{-1}\|p\|_{L^{2d/(d-1)}}
 \\ &{}
\le C((\|\bu\|_1\|\bu\|_2)^{1/2}+\nu^{-1}(\|p\|_0\|p\|_1)^{1/2})
\\
& {}
\le C\!\left(N_1(\bu,p)N_2(\bu,p)\right)^{1/2}.\ \
%C(\|\bu\|_1+\nu^{-1}\|p\|_0)^{1/2}(\|\bu\|_2+\nu^{-1}\|p\|_1)^{1/2}.
\end{align*}
\end{proof}
\begin{lema}\label{le:est_sh_inf_mu} There exist a positive constant~$C_1$ such that the following bounds hold
\begin{align}
\label{cota_sh_inf_mu}
\|\bs_h^m\|_\infty  & \le C_1(\|\bu\|_{d-2}\|\bu\|_2)^{1/2},
\\
\label{la_cota_mu}
\|\nabla\bs_h^m\|_{L^{2d/(d-1)}} & \le
 C_1\bigl(\|\bu\|_1\|\bu\|_2\bigr)^{1/2}.
\end{align}
\end{lema}
\begin{proof}
We  argue exactly as in the proof of Lemma~\ref{le:est_sh_inf} replacing \eqref{stokespro} by \eqref{stokespro_mod}.
\end{proof}
\begin{remark}\label{re:beta_cond}\rm
In the case $\mu=0$, according to~Lemma~\ref{le:est_sh_inf}, we have that $\beta\ge 8L$ when $\bw_h=\bs_h$ if for $t\ge 0$,
\begin{equation}\label{eq:as1_u}
\beta \ge 8\left(2\frac{\left(\textcolor{black}{(\hat c_1} C_0)^2N_1(\bu,p)N_2(\bu,p)\right)^{d/3}}{\nu^{(2d-3)/3}}+C_0^2\frac{
\|\bu\|_{d-2}\|\bu\|_2+N_1(\bu,p)N_{d-1}(\bu,p)}{\nu}\right),
\end{equation}
with $C_0$ the constant in~Lemma~\ref{le:est_sh_inf}.
In case $\mu\neq 0$ from \eqref{cotainfty1} and \eqref{cota_sh_inf_mu} we have that $\beta\ge 8L$  when $\bw_h=\bs_h^m$ if for $t\ge 0$
\begin{equation}\label{eq:as1_muno0_u}
\beta\ge 16\left(C_1\|\nabla \bu\|_\infty+C_1^2\frac{
\|\bu\|_{d-2}\|\bu\|_2}{4\mu} %+\frac{1}{2}
\right).
\end{equation}
\end{remark}
\subsection{The Lagrange interpolant}\label{sub_la}

In this section we consider when $I_H \bu=I_H^{La} \bu$.
With the help of the following lemmas we will show that
the analogous of Theorems~\ref{Th:main} and~\ref{Th:main_muno0} (Theorem~\ref{Th:main_la} below) also holds in this case.
\begin{lema}
Let $\bv_h\in X_{h,r}$ then the following bound holds
\begin{equation}\label{eq:cotainter_la2}
\|\bv_h-I_H^{La}\bv_h\|_0\le c_{La} H\|\nabla \bv_h\|_0,
\end{equation}
where
\begin{equation}\label{la_c_La}
c_{La}= C\left({H}/{h}\right)^{\frac{d(p-2)}{2p}},
\end{equation}
where $C$ is a generic constant and $p=3$ if $d=2$ and $p=4$ if $d=3$.
\end{lema}
\begin{proof}
For $\bv_h\in X_{h,r}$ we write
\begin{equation}
\| \bv_h - I_H^{La}\bv_h\|^2_0 = \sum_{K\in T_H} \| \bv_h - I_H^{La}\bv_h\|_{L^2(K)}^2 \le
 C\sum_{K\in T_H} H^{\frac{d(p-2)}{p}}\| \bv_h - I_H^{La}\bv_h\|_{L^p(K)}^2,
 \label{primera}
 \end{equation}
 the last inequality being a consequence of H\"older's inequality and of the fact that $| K| \le CH^d$.
 Applying \eqref{eq:cotainter_la} and \eqref{inv} we get
$$
\| \bv_h - I_H^{La}\bv_h\|_{L^p(K)} \le c_{\rm int} H\|\nabla\bv_h\|_{L^p(K)}\le c_{\rm int}c_{\rm inv}H\|\nabla\bv_h\|_{L^2(K)}h^{-\frac{d(p-2)}{2p}},
$$
so that inserting the above inequality into \eqref{primera} we reach \eqref{eq:cotainter_la2}.
\end{proof}
\begin{lema}\label{le:(I-I_h)(u-s_h)} Let $\bs_h$ be the Stokes projection defined in \eqref{stokesnew}. Then the following bound holds
\begin{equation}\label{eq:trun_la_stokes}
\|(I-I_H^{La})(\bs_h-\bu)\|_0\le C H^2 h^{r-2}\|\bu\|_r,
\end{equation}
\end{lema}
where $C$ is a generic constant.
\begin{proof}
We write
$$
(I-I_H^{La})(\bs_h - \bu) = (I-I_H^{La})(\bs_h - I_h^{La}\bu) +(I- I_H^{La})(I_h^{La}\bu-\bu)
$$
Applying \eqref{eq:cotainter_la2} and \eqref{la_c_La} to $\bv_h=\bs_h - I_h^{La}\bu$ and then \eqref{stokespro} and \eqref{eq:cotainter_la}
we get
\begin{align}\label{eq:primer_trun_la}
\|(I-I_H^{La})(\bs_h - I_h^{La}\bu)\|_0&\le  C  \left({H}/{h}\right)^{\frac{d(p-2)}{2p}}H \|\nabla (\bs_h - I_h^{La}\bu)\|_0
\nonumber\\&\le  C\left({H}/{h}\right)^{\frac{d(p-2)}{2p}} H h^{r-1}\|\bu\|_{r}\le C H^2 h^{r-2}\|\bu\|_r,
\end{align}
where in the last inequality we have bounded $(H/h)^{d(p-2)/(2p)}$ by $H/h$.
For the other term we argue as in \eqref{primera} and apply \eqref{eq:cotainter_la} to get
\begin{eqnarray}\label{eq:segun_trun_la}
\|(I- I_H^{La})(I_h^{La}\bu-\bu)\|_0^2&\le& C c_{\rm int} H^2 H^{\frac{d(p-2)}{p}}\sum_{K\in \tau_H}\|\nabla (I_h^{La}\bu-\bu)\|_{L^p(K)}^2\nonumber\\
&\le&C c_{\rm int} H^2 H^{\frac{d(p-2)}{p}}h^{2(r-2)}\sum_{K\in \tau_H} |\bu|_{r-1,p,K}^2.
\end{eqnarray}
Applying \eqref{sob1} with $s=1$ and taking into account $C H^d\le |K|\le C H^d$ we get
$
\|\bu\|_{L^p(K)}\le C H^{1-\frac{d(p-2)}{2p}}\|\bu\|_{1,K},
$
from which
$$
|\bu|_{r-1,p,K}^2 \le C H^{2-\frac{d(p-2)}{p}}\|\bu\|_{r,2,K}^2.
$$
Inserting the above inequality into \eqref{eq:segun_trun_la} we reach
\begin{eqnarray}\label{eq:tercer_trun_la}
\|(I- I_H^{La})(I_h^{La}\bu-\bu)\|_0\le C H^2 h^{r-2}\|\bu\|_{r}.
\end{eqnarray}
Finally, \eqref{eq:trun_la_stokes} follows from \eqref{eq:primer_trun_la} and \eqref{eq:tercer_trun_la}.
\end{proof}

\begin{Theorem}\label{Th:main_la} In the same conditions of Theorem~\ref{Th:main}
(resp.~Theorem~\ref{Th:main_muno0}), if $I_H$ is replaced by
$I_H^{La}$, $H$ satisfies condition~\eqref{eq:as2} with $c_I$ replaced by $c_{La}$ defined in \eqref{la_c_La},
and $H/h$ remains bounded, then, the statement of Theorem~\ref{Th:main} (resp.~Theorem~\ref{Th:main_muno0}) holds with $\gamma$ defined in \eqref{eq:gamma} with $c_I$ replaced by $c_{La}$.
\end{Theorem}
\begin{proof}
The proof of the theorem can be obtained arguing exactly as in the proof of Theorem~\ref{Th:main} (resp.~\ref{Th:main_muno0}) with only two differences that we
now state.
We first observe that assuming $H/h$ remains bounded we can apply \eqref{eq:cotainter_la2} instead of
\eqref{eq:cotainter} in~\eqref{eq:apply(21)a} and~\eqref{eq:apply(21)b}.
We also observe that since \eqref{eq:L^2inter} does
not hold for $I_H=I_H^{la}$ we cannot apply \eqref{eq:error5}. Instead,
 adding and subtracting $\bu-\bs_h$ and using \eqref{eq:trun_la_stokes} we get
\begin{align*}
&\beta|(I_H \bu-I_H \bs_h,I_H\be_h)|\le \beta |(I_H-I)(\bu-\bs_h),I_H\be_h)|+\beta|(\bu-\bs_h,I_H\be_h)|\nonumber\\
& \le \beta (C H^2 h^{r-2}\|\bu\|_r+\|\bu-\bs_h\|_0)\|I_H\be_h\|_0\le \beta (C h^r\|\bu\|_r+\|\bu-\bs_h\|_0)\|I_H\be_h\|_0,
\end{align*}
where in the last inequality we have applied that since $H/h$ is bounded then $H\le C h$.
Then we replace \eqref{eq:error5} in the proof of Lemma~\ref{lema_general} and consequently in the proof of Theorem~\ref{Th:main} (resp.~\ref{Th:main_muno0}). by the following inequality
$$
\beta|(I_H \bu-I_H \bs_h,I_H\be_h)|\le \frac{\beta}{2}(C h^r\|\bu\|_r+\|\bu-\bs_h\|_0)^2+\frac{\beta}{2}\|I_H \be_h\|_0^2
$$
and we can conclude applying the same arguments.
\end{proof}

\section{Numerical experiments}\label{se:num}

We check the results of the previous section with some numerical experiments. As it is customary for these purposes, we use an example with a known solution. In particular, we consider
 the Navier-Stokes equations in $\Omega=[0,1]^2$, with the forcing term~$\bff$ chosen so that the solution $\bu$ and~$p$ are given by
%bx=8*(sin(pi*x).^2).*(2*y.*(1-y).*(1-2*y)).*(0.5*(1.2+0.8*cos(4*t)));
%by=-16*pi*((y.*(1-y)).^2).*sin(pi*x).*cos(pi*x).*(0.5*(1.2+0.8*cos(4*t)));
%    p=sin(pi*x).*cos(pi*y).*(0.5*(1.2+0.8*cos(4*t)));
\begin{eqnarray}
\label{eq:exactau}
\bu(x,y,t)&=& \frac{6+4\cos(4t)}{10} \left[\begin{array}{c} 8\sin^2(\pi x) (2y(1-y)(1-2y)\\
-8\pi\sin(2*\pi x) (y(1-y))^2\end{array}\right]\\
p(x,y,t)&=&\frac{6+4\cos(4t)}{10} \sin(\pi x)\cos(\pi y).
\label{eq:exactap}
\end{eqnarray}
For the spatial discretization we used $P_2/P_1$ elements on a regular triangulation with SW-NE diagonals, with
the same number of subdivisions on each coordinate direction. For for coarse mesh interpolation we take piecewise
constants. The time integration was done with an implicit/explicit (IMEX) method based on the second order
backward differentiation formula (BDF), where, to avoid solving nonlinear steady problems at each step, linear extrapolation of the form
$
b_h(2\bu_h(t-\Delta t)-\bu_h(t-2\Delta t),\bu_h(t),\bvar_h)
$
was used in the convection term, where $\Delta t$ is the time step, except in the first step where
$b_h(\bu_h(t-\Delta t),\bu_h(t),\bvar_h)$ was used. The time step was chosen so that the error arising from the spatial
discretization was dominant. To check that this was the case, we made sure that results were not essentially altered
if recomputed with a smaller $\Delta t$. Unless stated otherwise, in what follows the initial condition was set to~$\bu_h={\bf 0}$ and~$p=0$, so that there is
an $O(1)$ error at time $t=0$.

We first check that there is no upper bound on the nudging parameter~$\beta$. The left plot in~Fig.~\ref{fig0} shows the velocity
errors in~$L^2$ vs time for different values of~$\beta$ for $\nu=10^{-6}$, including~$\beta=100$. It can be seen a clear difference
between $\beta=0$, where the initial errors do not decay with time, and $\beta>0$ where they do, and for the four
largest values of~$\beta$ shown, they do so exponentially in time, until an asymptotic regime is reached.
 We also notice
that the results are little altered for $\beta\ge 10$.
%The results corresponding to~$\beta=1000$ are plotted in discontinuous line  because they are almost superimposed to those of~$\beta=100$.
\begin{figure}[h]
\label{fig0}
\begin{center}
\includegraphics[height=5.1truecm]{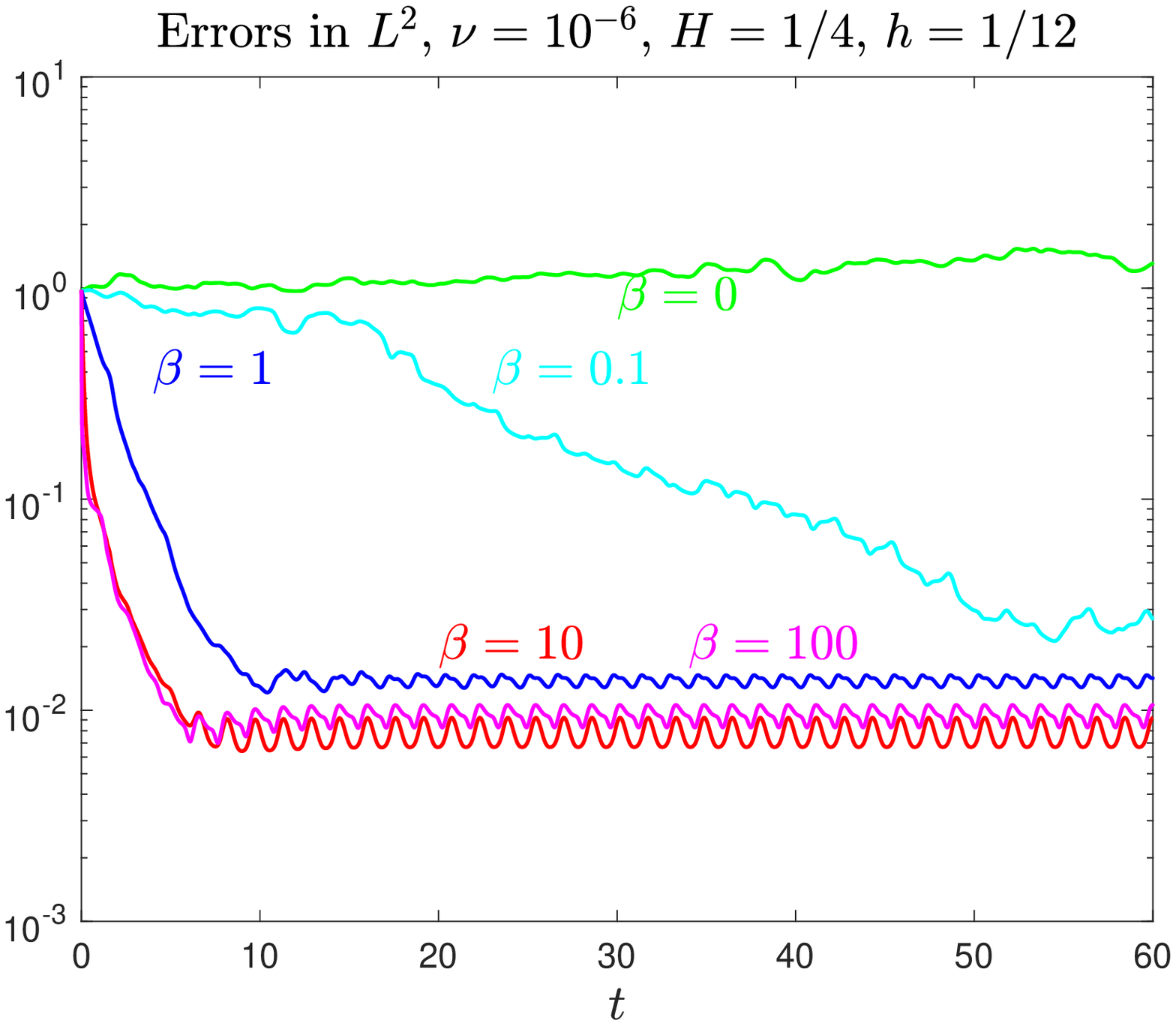}\,
\includegraphics[height=5.1truecm]{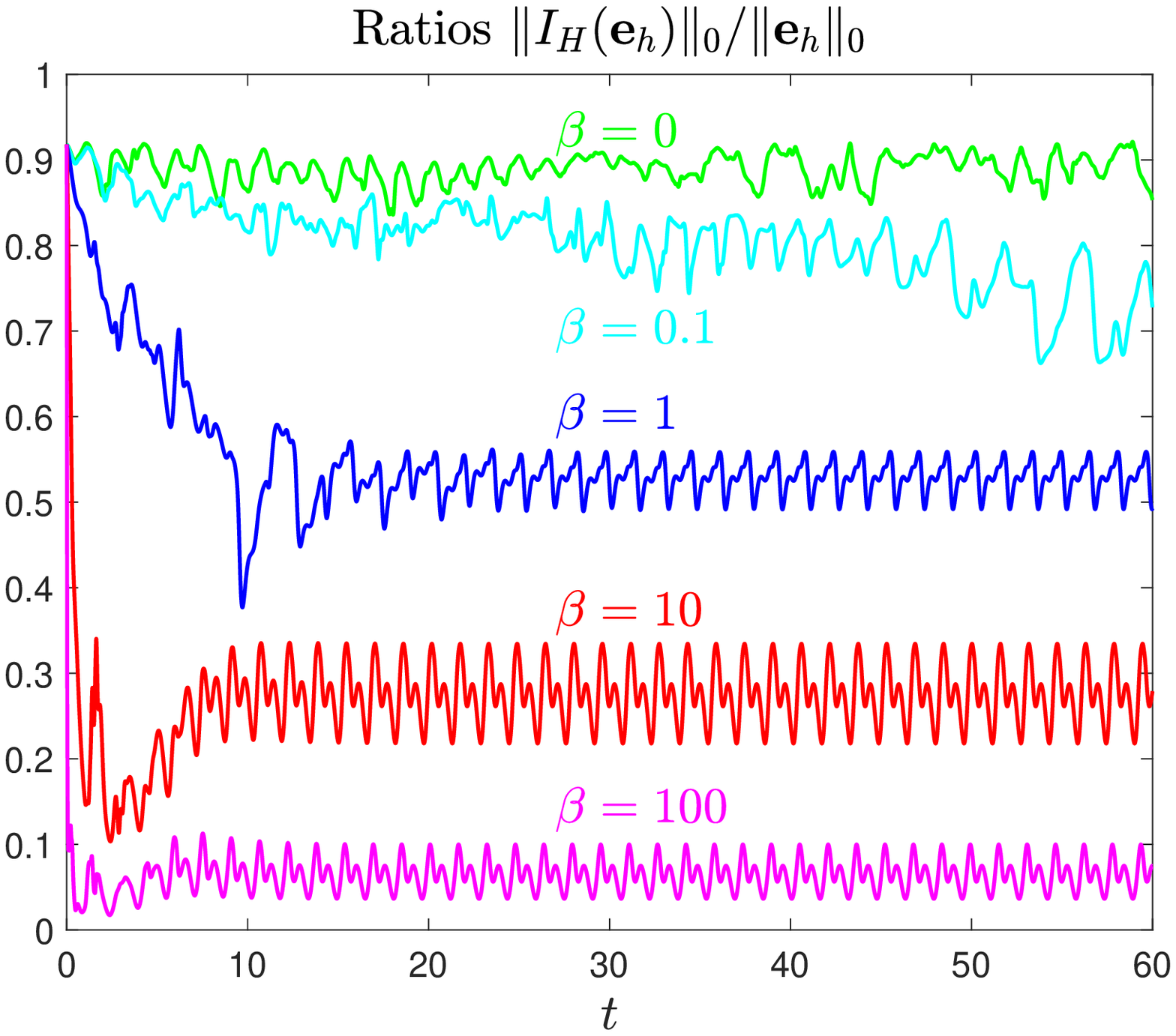}
\end{center}
\caption{Velocity errors vs time}
\end{figure}
In view of~(\ref{eq:gamma}), one may be temped to question the advantage of taking~$\beta\ge \nu(c_IH)^{-2}$,  since the rate of decay of the initial errors, $\gamma$,
is unaltered for larger values of~$\beta$. If we assume that $c_I$ is of order one, then,
the value of $ \nu(c_IH)^{-2}$ in the present example is unlikely
to be larger than $10^{-5}$, so that Fig.~\ref{fig0} (and more examples in~\cite{Larios_et_al} and~\cite{Rebholz-Zerfas})
seems to suggest that there is some advantage in taking $\beta\ge \nu(c_IH)^{-2}$ if we want a faster decay of
the initial errors. Since this is in apparent contradiction with the analysis in the previous section, we now propose
an alternative explanation.

Let us consider  for some integer $k\ge 2$ the value $r=\|I_H(\bu(0))\|_0/(k\|\bu(0)\|_0)$. If we take the intial condtion $\bu_h=0$, then,
by continuity there exist $t_0>0$ such that
\begin{equation}
\label{laineg}
\|I_H(\be_h(t)\|_0\ge r\|\be_h(t)\|_0,\qquad t\in[0,t_0].
\end{equation}
Consequently, from (\ref{eq:otra}) it follows that
\begin{align*}
\frac{d}{dt}\|\be_h\|_0^2+(\beta r^2-2L)\| \be_h\|_0^2
\le &
2\overline k \|\tau_h^2\|_0^2+{\beta} c_0^2\|\bu-\bw_h\|_0^2\\
&{}+\left(\!(1-\overline \mu)\frac{2\hat c_P}{\nu}+{\overline \mu}{L}\!\right)\!\|\btau_h^1\|_{-1+\overline \mu}^2,
\end{align*}
for $t\in[0,t_0]$, which, for $\beta> 2L/r^2$, could explain that initial errors decay faster when larger values of~$\beta$ are taken.
In Fig.~\ref{fig0} we also show the ratios~$\|I_H(\be_h)\|_0/\|\be_h\|_0$. It can be seen that although they became smaller
as $\beta$ is increased, they are sufficiently away from zero to suggest that the analysis in the present section
may explain the faster rates of decay of the initial errors when larger values of~$\beta$ are taken.

%In what follows, we show errors corresponding to the asymptotic regime. More precisely, we show
%\begin{equation}
%\label{errsu_en_linfty}
%\max_{t_a\le t\le t_f}\|\bu_h(t)-\bu(t)\|_0,
%\end{equation}
%for $t_a$ a value of time after
%the asymptotic behaviour has shown itself (well after errors cease to decay exponentially in time) and $t_f$ is the final time in the computations.

In the reminder of this section we take~$\beta=1$.
Since, as shown in Fig.~\ref{fig0}, after an initial decay, the errors show an oscillatory behaviour, in the examples below by $L^2$ errors we mean
the maximum of errors $\|\bu_h(t)-\bu(t)\|_0$ for values of $t$ after the asymptotic regime has shown itself.

We now check the rates of convergence proved in the present paper.
In Fig.~\ref{fig1} we show errors vs~$h$ for different values of the diffusion parameter~$\nu$ and compare
the cases of positive $\mu$ ($\mu=0.05$) and~$\mu=0$.
The value of~$H$ is~$H=3h$ and $\beta$ is set to~$\beta=1$.
%For the computation of the errors, we took $t_a=40$ and~$t_f=60$.
Results corresponding to the smallest value of~$\nu$ are represented with
discontinuous lines in both plots so that they can be seen superimposed to those corresponding to larger values of~$\nu$.  Slopes of least squares fits to the results corresponding to each value of~$\nu$ are shown, so that
the order of convergence can be checked. In both cases, $\mu=0.05$ and~$\mu=0$, $O(h^3)$ errors are obtained for large values of~$\nu$, which is what Theorem~\ref{Th:main} and~Remark~\ref{re:nu_large} predict.
However, for smaller values of~$\nu$, while the errors with~positive $\mu$ become $O(h^2)$ and independent of~$\nu$,
as Theorem~\ref{Th:main_muno0} predicts,
for~$\mu=0$ the method does not have convergent behaviour for the values of~$h$ shown (presumably, the method will show
convergence for $h\le \nu$).
\begin{figure}[h]
\label{fig1}
\begin{center}
\includegraphics[height=5.1truecm]{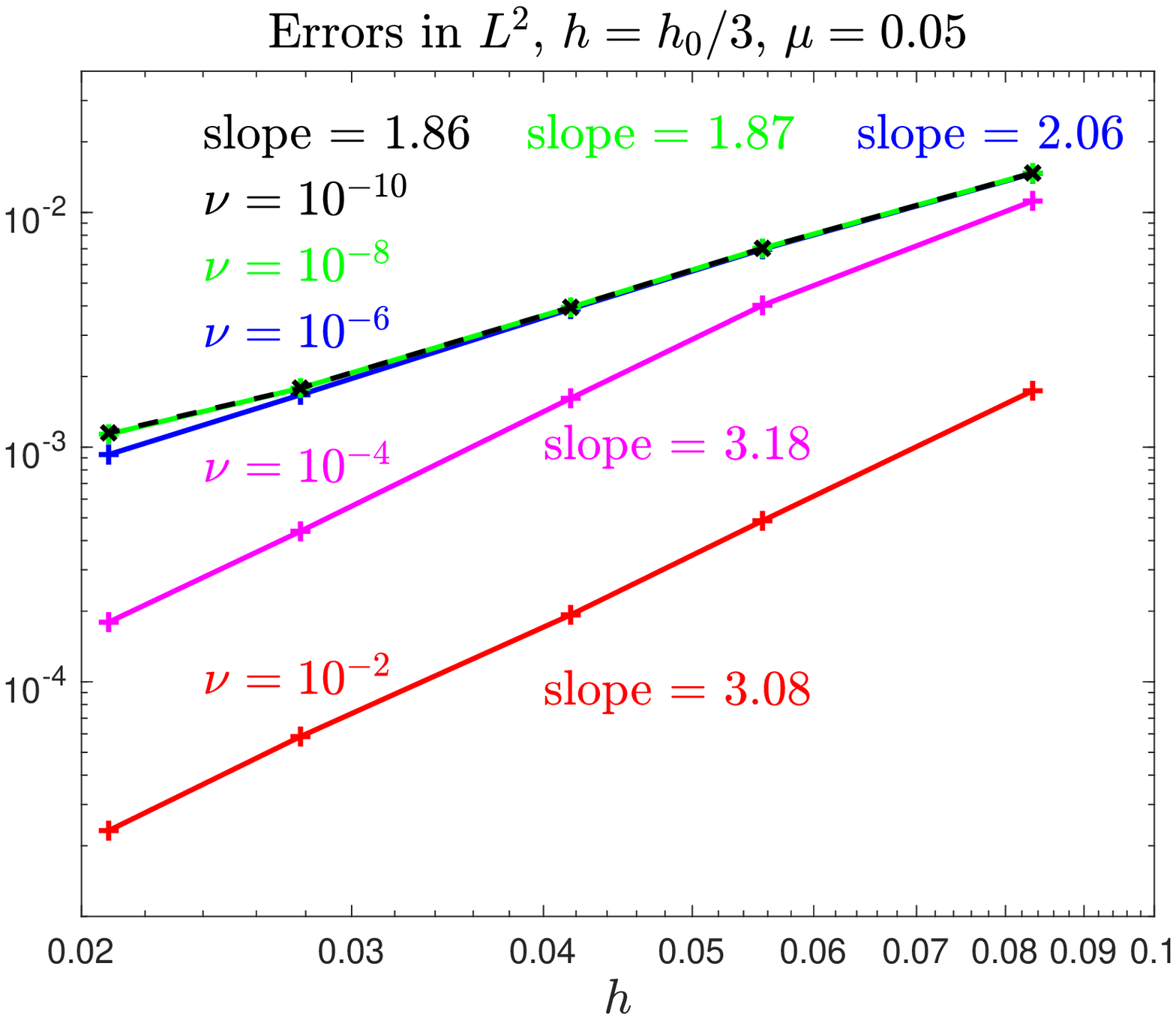} \, \includegraphics[height=5.1truecm]{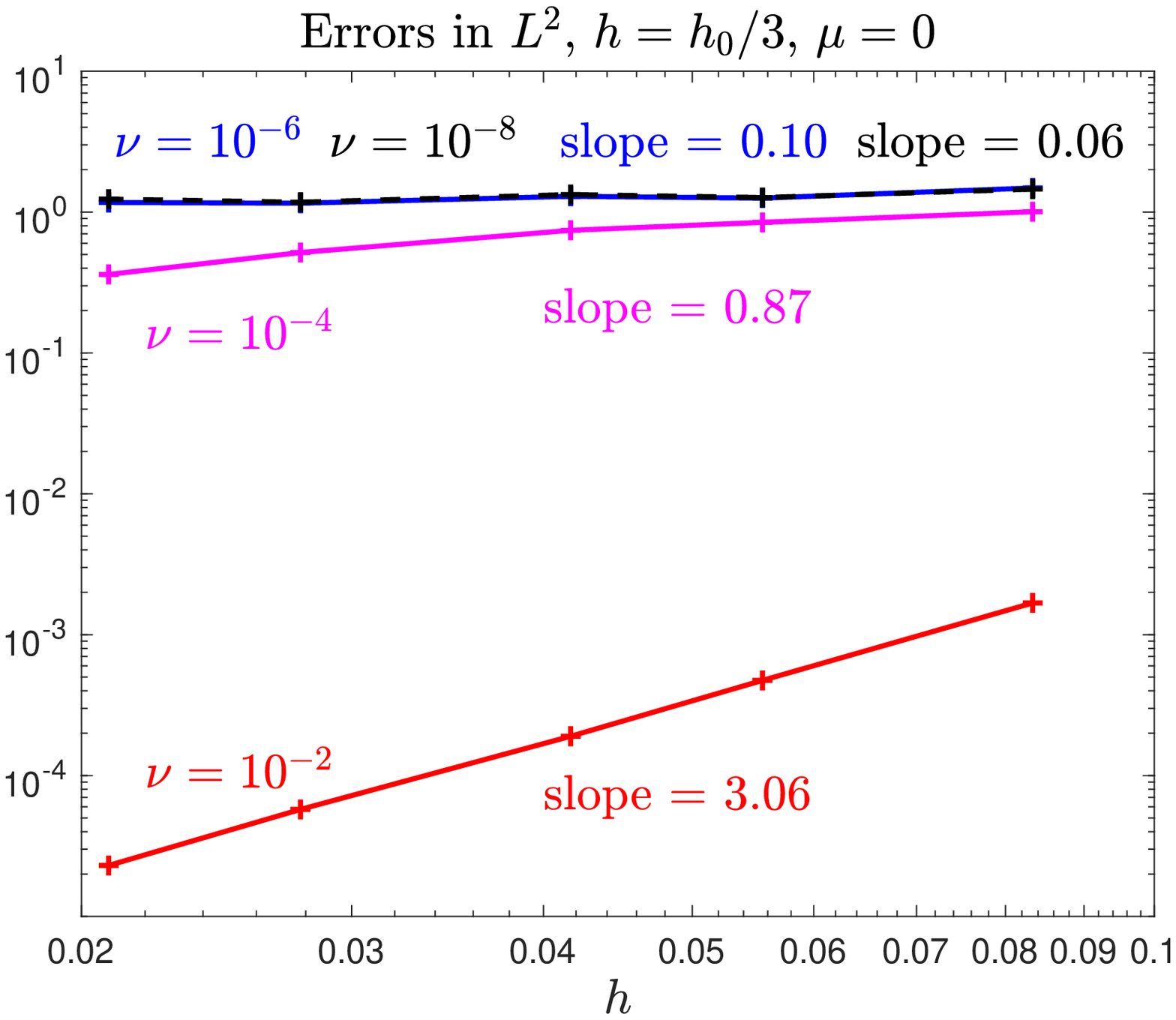}
\end{center}
\caption{Velocity errors. Left, $\mu=0.05$ Right, $\mu=0$.}
\end{figure}

Finally, we check that the requirement~$H/h$ bounded is required for convergence if Lagrange interpolants are used.
In Fig.~\ref{fig2} we show velocity errors when $h\rightarrow 0$ in two different scenarios:  $h=H/3$ (left) and $H$
fixed to~$H=0.25$. We see that while $H=3h$, the method converge as predicted by Theorem~\ref{Th:main_la} (the value
of~$\nu=10^{-6}$ and that of~$\mu=0.05$). If $H$ is kept fixed, however, the method using the Lagrange Interpolant
does not exhibit convergent behaviour.
We remark, however, that with larger values of $\beta$ or~$\nu$, convergence is not altered as much as in Fig.~\ref{fig2} when $H/h$ grows.
% is not as dramatic as in~Fig.~\ref{fig2}, and we hardly noticed it for $\nu\ge 10^{-4}$.
Nevertheless, this example shows the risks of not keeping $(H/h)$ bounded with Lagrange interpolants.

%Next, we check that the requirement~$H/h$ bounded is required for convergence if Lagrange interpolants are used.
%In Fig.~\ref{fig2} we show velocity errors when $h\rightarrow 0$ in two different scenarios:  $h=H/3$ (left) and $H$
%fixed to~$H=0.25$. We see that while $H/h$ is kept fixed, the methods converge according to the theory (the value
%of~$\nu=10^{-6}$ and that of~$\mu=0.05$). If $H$ is kept fixed, however, the method using the Lagrange Interpolant
%does not exhibit convergent behaviour. The slope of~$0.78$ shown in the plot corresponds to a least-squares fit
%to the results of the three largest values of~$h$. This is in marked contrast with the method using the Scott-Zhang
%interpolant, which seems to be little affected by keeping~$H$ fixed.
%%Let us observe that, in the case of the Lagrange interpolant,
%keeping $H$ fixed not only affects the order
%of decay with~$h$ of some terrms in the error bound, as Lemma~\ref{le:(I-I_h)(u-s_h)} shows, but also
%when applying~\eqref{eq:cotainter_la2} instead of
%\eqref{eq:cotainter} in~\eqref{eq:apply(21)a} and~\eqref{eq:apply(21)b}
%
\begin{figure}
\label{fig2}
\begin{center}
\includegraphics[height=5.1truecm]{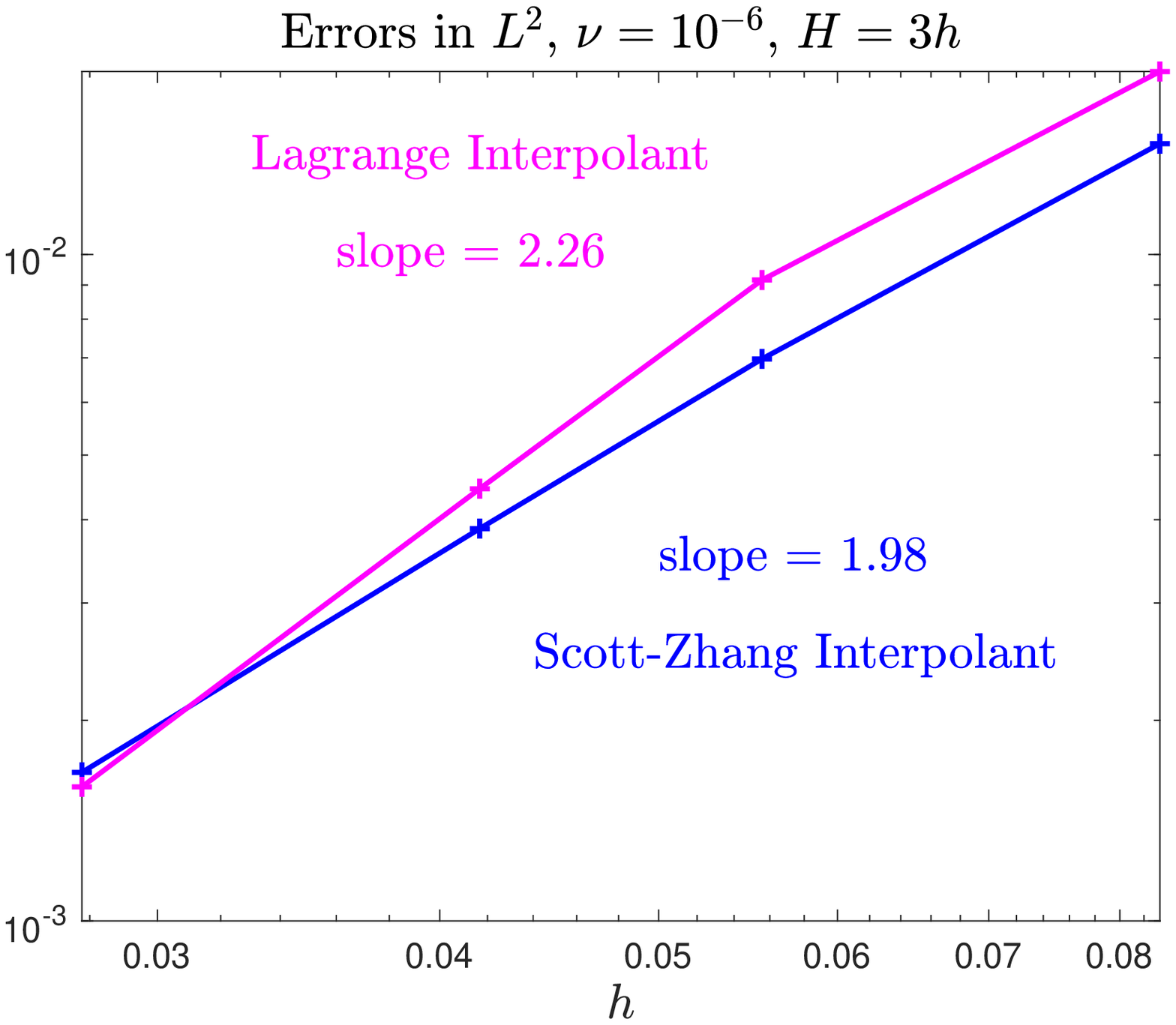} \, \includegraphics[height=5.1truecm]{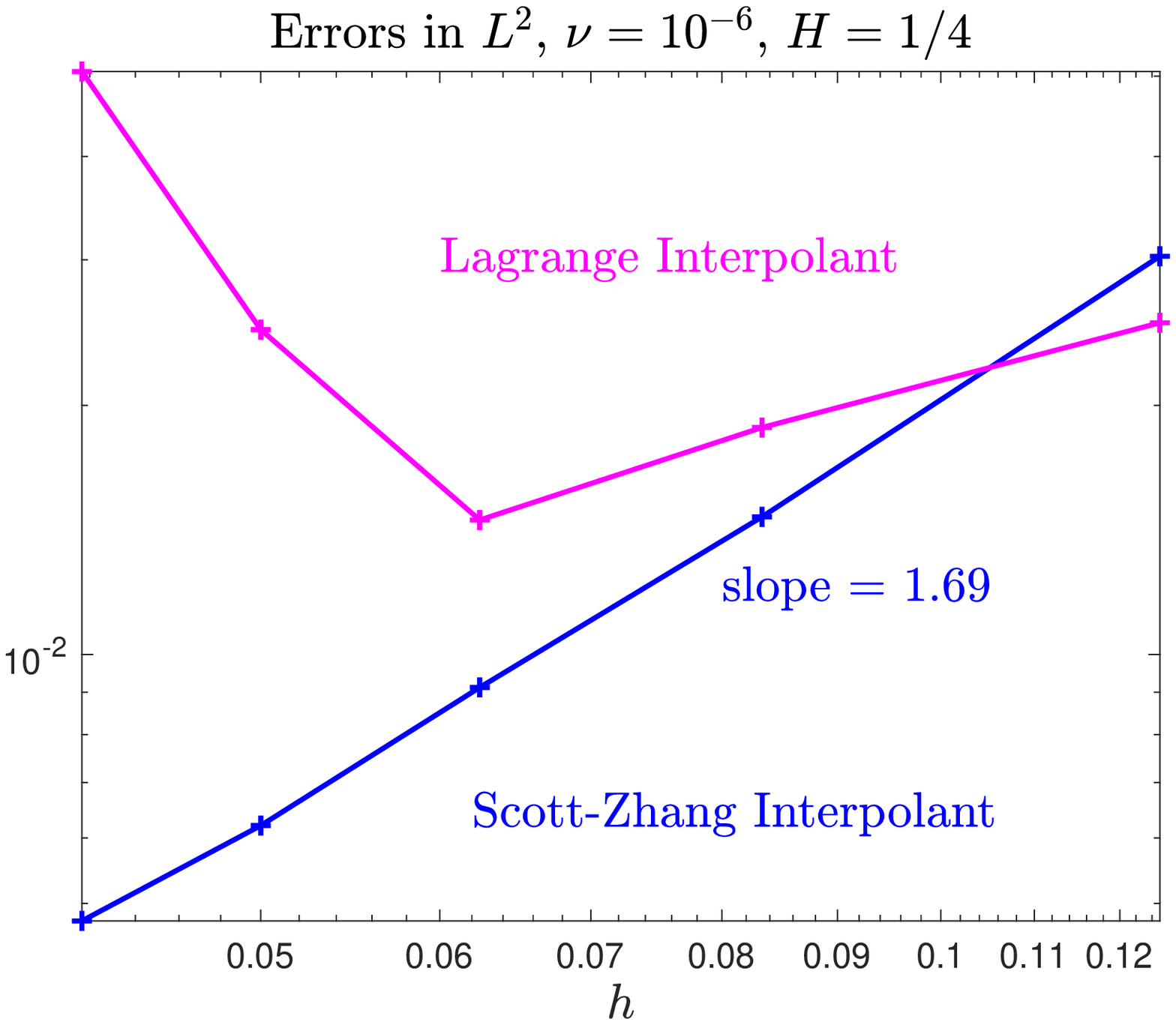}
\end{center}
\caption{Velocity errors. Left $H=3h$. Right $H=0.25$.}
\end{figure}

\section{Conclusions}  We have analyzed a semidiscretization in space by inf-sup stable mixed finite elements of a  continuous downscaling data assimilation method for the two and three-dimensional Navier-Stokes equations.  The data assimilation method, introduced in~\cite{Az_Ol_Ti},  combines observational data (measurements) on  large spatial scales or coarse mesh, $I_H\bu$,  with
simulations in order to improve predictions of the physical phenomenon  being studied.  We have considered the Galerkin method with and without grad-div stabilization. Uniform error bounds in time have been obtained for
the approximation to velocity field, under standard assumptions in finite element analysis. The order of convergence proved for the method wthout stabilizationi is optimal, in the sense that it is the best that can be obtained with the finite element space
being used (i.e., errors of the same order as interpolation). For the Galerkin method with grad-div stabilization error bounds in which the constants are independent on inverse powers of the viscosity are proved.  Convergence rates and dependence or independence of~$\nu$ are corroborated in numerical experiments. As opposed to previous works in the literature,
our analysis also covers the case in which $I_H \bu$ is the standard Lagrange interpolant, where we show that~$H/h$ must be kept bounded in order to get convergence. Also, the upper bound on the nudging parameter assumed in previous references is removed. The techniques of analysis used in the present paper allow to improve the available error bounds for a closely-related finite element method in~\cite{Larios_et_al}.
\bibliographystyle{abbrv}

\bibliography{references}

\end{document}

%% file: data_assim_shared.tex
\usepackage{lipsum}
\usepackage{amsfonts}
\usepackage{graphicx}
\usepackage{epstopdf}
\usepackage{algorithmic}
\ifpdf
  \DeclareGraphicsExtensions{.eps,.pdf,.png,.jpg}
\else
  \DeclareGraphicsExtensions{.eps}
\fi

\newtheorem{Theorem}{Theorem}[section]
\newtheorem{lema}[Theorem]{Lemma}

\newtheorem{remark}[Theorem]{Remark}

\newtheorem{Proof}{{\em Proof:}}

\newenvironment{proof}{\begin{Proof}\rm}{\hfill $\Box$ \end{Proof}}

\title{Uniform in time error estimates for a finite element method applied to a downscaling data assimilation algorithm
for the Navier--Stokes equations
}
\author{ Bosco
Garc\'{i}a-Archilla\thanks{Departamento de Matem\'atica Aplicada
II, Universidad de Sevilla, Sevilla, Spain. Research is supported by
Spanish MINECO under grant MTM2015-65608-P (bosco@esi.us.es).}
  \and Julia Novo\thanks{Departamento de
Matem\'aticas, Universidad Aut\'onoma de Madrid, Spain.  Research is supported
by Spanish MINECO
under grant MTM2016-78995-P (AEI/FEDER, UE) and VA024P17 (Junta de Castilla y Leon, ES) cofinanced by FEDER funds (julia.novo@uam.es).}
\and Edriss S. Titi\thanks{Department of Mathematics, Texas A\&M University, College Station, TX 77843, USA. Department of Applied Mathematics and Theoretical Physics, University of Cambridge, Wilberforce Road, Cambridge CB3 0WA, UK.  Department of Computer Science and Applied Mathematics, Weizmann Institute of Science, Rehovot 76100, Israel. Research is supported in part by the ONR grant N00014-15-1-2333, the Einstein Stiftung/Foundation - Berlin,
through the Einstein Visiting Fellow Program, and by the John Simon Guggenheim Memorial Foundation (titi@math.tamu.edu \quad  Edriss.Titi@damtp.cam.ac.uk).}}

\usepackage{amsopn}